\documentclass[11pt,reqno]{amsart}
\usepackage[utf8]{inputenc}

\usepackage{amssymb, amsmath, amsthm}
\usepackage[bookmarks, bookmarksdepth=2, colorlinks=true, linkcolor=blue, citecolor=blue, urlcolor=blue]{hyperref}

\usepackage[alphabetic,lite]{amsrefs}
\usepackage{verbatim}
\usepackage{amscd}   
\usepackage[all]{xy} 
\usepackage{youngtab} 
\usepackage{young} 
\usepackage{ytableau}
\usepackage{tikz}
\usepackage{ mathrsfs }
\usepackage{cases}
\usepackage{array}
\usepackage{cellspace}
\usepackage{calligra,mathrsfs}
\usepackage{bm}
\usepackage{graphicx}
\usepackage{rank-2-roots}
\usepackage{float}
\usepackage{enumitem}

\usepackage[prependcaption,textsize=scriptsize]{todonotes}

\newcommand{\defi}[1]{{\upshape\sffamily #1}}

\DeclareMathOperator{\ShHom}{\mathscr{H}\text{\kern -3pt {\calligra\large om}}\,}

\def\kk{{\mathbf k}}
\renewcommand{\ll}{\lambda}

\newcommand{\oo}{\otimes}

\newcommand{\bG}{{\bf G}}
\newcommand{\bh}{{\bf h}}
\newcommand{\bN}{{\bf N}}
\newcommand{\bW}{{\bf W}}

\newcommand{\avg}{\operatorname{avg}}

\newcommand{\GL}{\operatorname{GL}}

\newcommand{\Sym}{\operatorname{Sym}}

\newcommand{\bb}[1]{\mathbb{#1}}

\newcommand{\mc}[1]{\mathcal{#1}}
\newcommand{\mf}[1]{\mathfrak{#1}}
\newcommand{\ol}[1]{\overline{#1}}
\newcommand{\op}[1]{\operatorname{#1}}

\newcommand{\ul}[1]{\underline{#1}}

\def\PP{{\mathbf P}}
\def\lra{\longrightarrow}

\newtheorem{theorem}{Theorem}[section]
\newtheorem*{theorem*}{Theorem}
\newtheorem*{problem*}{Problem}
\newtheorem{lemma}[theorem]{Lemma}

\newtheorem{proposition}[theorem]{Proposition}
\newtheorem{corollary}[theorem]{Corollary}
\newtheorem*{corollary*}{Corollary}

\newtheorem*{main-thm*}{Main Theorem}
\newtheorem*{linear-resolutions*}{Theorem on Linear Resolutions}
\newtheorem*{regularity-powers*}{Theorem on Regularity}
\newtheorem*{injectivity-Ext*}{Theorem on Injectivity of Maps of Ext Modules}
\newtheorem*{Kodaira*}{Kodaira Vanishing for Determinantal Thickenings}

\theoremstyle{definition}
\newtheorem{definition}[theorem]{Definition}
\newtheorem*{definition*}{Definition}
\newtheorem{example}[theorem]{Example}

\theoremstyle{remark}
\newtheorem{remark}[theorem]{Remark}
\newtheorem*{remark*}{Remark}

\numberwithin{equation}{section}

\begin{document}

\title{Lefschetz properties for monomial complete intersections}

\author{Annet Kyomuhangi}
\address{Department of Mathematics, Busitema University, P.O. Box 236, Tororo}
\email{annet.kyomuhangi@gmail.com}

\author{Emanuela Marangone}
\address{Dipartimento di Matematica, Università degli Studi di Genova, Via Dodecaneso, 35, 16146 Genova, Italy \newline
\indent CIMAT - Centro de Investigación en Matemáticas, Valenciana, 36023 Guanajuato, Mexico}
\email{emanuela.marangone@cimat.mx}

\author{Claudiu Raicu}
\address{Department of Mathematics, University of Notre Dame, 255 Hurley, Notre Dame, IN 46556\newline
\indent Institute of Mathematics ``Simion Stoilow'' of the Romanian Academy}
\email{craicu@nd.edu}

\author{Ethan Reed}
\address{Academy of Mathematics and Systems Science, Chinese Academy of Sciences, No. 55 Zhongguancun
East Road, Beijing, 100190, China.}
\email{ethan.reed@amss.ac.cn}


\date{\today}


\begin{abstract} 
We give a complete characterization of the weak Lefschetz property (WLP) for monomial complete intersections over a field of positive characteristic. Richard Stanley observed that in characteristic zero WLP, and in fact the strong Lefschetz property (SLP), holds for all degree sequences, as a consequence of the hard Lefschetz theorem, and the same result was explained by Junzo Watanabe using the representation theory of~$\mf{sl}_2$. In positive characteristic, many partial results are known, most notably the classification for constant degree sequences due to Brenner--Kaid and Kustin--Vraciu. Our approach is based on a cohomological and representation-theoretic interpretation of WLP. Combined with an analysis of cohomology characters, this leads to a complete numerical criterion for WLP, expressed by simple inequalities involving the $p$-adic digits of the exponents. We also give a new proof of the known classification of SLP using Renaud’s algorithm for multiplication in the Green--Han--Monsky ring.
\end{abstract}

\maketitle

\section{Introduction}

Let $\kk$ be a field of characteristic $p$, and for a sequence $\ul{a}\in\bb{Z}^n_{\geq 0}$ consider the monomial complete intersection
\begin{equation}\label{eq:def-Ma} 
M_{\ul{a}} = \kk[T_1,\ldots,T_n]/\langle T_1^{1+a_1},\ldots,T_n^{1+a_n}\rangle.
\end{equation}
We say that $M_{\ul{a}}$ has the \defi{weak Lefschetz property (WLP)} if there exists a linear form $\ell$ such that for all $e\geq 0$ the multiplication by $\ell$, 
\begin{equation}\label{eq:timesT-maxrk}
\times \ell: (M_{\ul{a}})_e \lra (M_{\ul{a}})_{e+1},
\end{equation}
has maximal rank. $M_{\ul{a}}$ has the \defi{strong Lefschetz property (SLP)} if for some linear form $\ell$ and for all $d,e\geq 0$
\[\times \ell^d: (M_{\ul{a}})_e \lra (M_{\ul{a}})_{e+d}\]
has maximal rank. A well-known observation of Stanley is that in characteristic zero, $M_{\ul{a}}$ satisfies WLP and SLP \cites{stanley,watanabe}, which follows from the hard Lefschetz theorem applied to the product of projective spaces $\PP^{a_1}\times\cdots\times\PP^{a_n}$. In characteristic $p>0$ this is no longer true, and a characterization of the complete intersections $M_{\ul{a}}$ for which SLP holds was obtained in \cites{nick,lund-nick, Cook12}, but the corresponding classification for WLP has remained open until now. There are many partial results in the literature regarding WLP \cites{bre-kai,kus-vra,Vraciu,Cook12}, and the goal of our work is to provide a complete characterization of the sequences $\ul{a}$ for which $M_{\ul{a}}$ has WLP, illustrate it with several applications, and explain how it recovers some of the most important prior results. In the final section we also explain a different take on the classification of SLP.

To state our main theorem we need to establish some preliminary notation. For $\ul{m}=(m_1,\cdots,m_n)\in\bb{Z}^n_{\geq 0}$ we write $|\ul{m}|=m_1+\cdots+m_n$, and define the \defi{gap $g(\ul{m})$} and the \defi{nonnegative gap $g_+(\ul{m})$} via
\begin{equation}\label{eq:defgap-g+} 
g(\ul{m}) = 2\max(m_i) - |\ul{m}|\quad\text{and}\quad g_+(\ul{m}) = \begin{cases}
\max(g(\ul{m}),0) & \text{if }|\ul{m}|\text{ is even}, \\
\max(g(\ul{m}),1) & \text{if }|\ul{m}|\text{ is odd}. \\
\end{cases}
\end{equation}
A positive value for $g(\ul{m})$ measures the extent to which $\ul{m}$ fails to be balanced (or satisfy the generalized triangle inequality). In Corollary~\ref{cor:g+m=mina-b} we show that $g_+(\ul{m})$ measures the minimal distance $a-b$ between the parts of a $2$-row partition $(a,b)$ such that $\ul{m}$ appears as a weight in the Schur polynomial $s_{(a,b)}$.

\begin{definition}\label{def:reduction+gap}
Fix a sequence $\ul{a}=(a_1,\ldots,a_n)$ and a prime power $q=p^k$.
\begin{enumerate}
    \item We say that $\ul{r}=(r_1,\ldots,r_n)$ is a \defi{$q$-reduction} of $\ul{a}$ if 
    \[a_i\equiv r_i\ (\op{mod}\ q)\text{ and }0\leq r_i\leq\min(a_i,2q-1)\text{ for all }i=1,\cdots,n.\]
    \item For a $q$-reduction $\ul{r}$, we set $m_i = (a_i-r_i)/q$ and define its \defi{gap} and \defi{nonnegative gap}
    \[g(\ul{a};\ul{r}) = g(\ul{m})\quad\text{and}\quad g_+(\ul{a};\ul{r}) = g_+(\ul{m}).\]
    We say that the $q$-reduction $\ul{r}$ is \defi{balanced} if $g(\ul{a};\ul{r})\leq 0$. It is even if $|\ul{m}|$ is even, and it is odd otherwise.
\end{enumerate}
\end{definition}

\noindent Notice that for fixed $q$, the number of $q$-reductions of $\ul{a}$ is at most $2^n$. Moreover, if $q>\max(a_i)$ then there exists a unique $q$-reduction, namely $\ul{r} = \ul{a}$, and this reduction is balanced and even.

For $q=p^k$ and $a\in\bb{Z}$ we write $a = 2qm + r$ with $0\leq r<2q$, and define 
\begin{equation}\label{eq:def-theta-q}
 \theta_q(a) = \begin{cases}
     r & \text{if }r\leq q-1,\\
     2q-2-r & \text{if }q-1\leq r\leq 2q-2,\\
     -\infty & \text{if }r=2q-1.
 \end{cases}
\end{equation}
As noted in Section~\ref{subsec:tensor-truncated}, we can view $\theta_q(a)$ as the largest imbalance $|u-v|$ obtained when writing the residue $r$ as $u+v$, with $0\leq u,v\leq q-1$.

\begin{theorem}\label{thm:new-WLP}
    The monomial complete intersection $M_{\ul{a}}$ satisfies WLP if and only if for every $q=p^k$, and for every $q$-reduction $\ul{r}$ of $\ul{a}$ we have
    \begin{equation}\label{eq:thetaq-leq-gap}
    \sum_{i=1}^n\theta_q(r_i)\leq \left(g_+(\ul{a};\ul{r})+2\right)q-2.
    \end{equation}
    Moreover, it is enough to verify \eqref{eq:thetaq-leq-gap} for those $q$ that satisfy
\begin{equation}\label{eq:bounds-q-ineq-WLP}
p \leq q \leq \frac{a_1+\cdots+a_n+1}{2}.
\end{equation} 
\end{theorem}

It is well-known that for $n=2$, WLP holds for all degree sequences $\ul{a}$ \cites{HMNW,MZ}. To obtain this as a consequence of Theorem~\ref{thm:new-WLP}, it suffices to combine the inequalities $\theta_q(r_i)\leq q-1$ and $g_+(\ul{a};\ul{r})\geq 0$. Without loss of generality, we will therefore assume that $n\geq 3$, and moreover $a_1\geq\cdots\geq a_n$. It is perhaps not immediately clear that if $a_1\geq a_2+\cdots+a_n$ then the conditions \eqref{eq:thetaq-leq-gap} hold (see Lemma~\ref{lem:gap-conditions-unbalanced}), so Theorem~\ref{thm:new-WLP} recovers the elementary statement that WLP holds for $M_{\ul{a}}$ under this hypothesis (see also \cite{mig-mr}*{Proposition~5.2}). We may therefore assume in addition that $a_1<a_2+\cdots+a_n$. The inequality \eqref{eq:thetaq-leq-gap} is automatic for $q=1$, and for an even balanced reduction $\ul{r}$ it becomes
\begin{equation}\label{eq:thetaq-leq-balanced}
\sum_{i=1}^n\theta_q(r_i)\leq 2q-2.
\end{equation}
In particular, if $q>\max(a_i)$ so that $\ul{r}=\ul{a}$ is the unique $q$-reduction, \eqref{eq:thetaq-leq-balanced} becomes $\sum_{i=1}^n a_i \leq 2q-2$, which explains why the only values of $q$ that need to be considered in Theorem~\ref{thm:new-WLP} are those for which \eqref{eq:bounds-q-ineq-WLP} holds.

\begin{example}\label{ex:2p-1s}
    Suppose that $p,n\geq 3$, and let $\ul{a}=(2p-1,2p-1,2p-1,1^{n-3})$. Theorem~\ref{thm:new-WLP} implies that WLP holds for $M_{\ul{a}}$ if and only if $n=3,4$. Indeed, for the $p$-reduction $\ul{r}=(p-1,p-1,p-1,1^{n-3})$ we have 
    \[m_i = (a_i-r_i)/p=\begin{cases}
        1 & i=1,2,3,\\
        0 & i>3.
    \end{cases}\]
    Since $|\ul{m}|=3$ is odd and $g(\ul{m}) = -1$, we get $g_+(\ul{a};\ul{r})=1$ and the inequality \eqref{eq:thetaq-leq-gap} becomes
    \[3p+n-6 = \sum_{i=1}^n\theta_p(r_i)\leq 3p-2,\]
    hence WLP fails for $n\geq 5$. We may therefore assume $n\leq 4$, so the above condition is satisfied. Since the inequality \eqref{eq:bounds-q-ineq-WLP} fails for $q\geq p^2\geq 3p$, it is enough to consider the conditions \eqref{eq:thetaq-leq-gap} when $q=p$. Besides the $p$-reduction of $\ul{a}$ we already considered, which satisfies \eqref{eq:thetaq-leq-gap} because $n\leq 4$, every other $p$-reduction $\ul{r}$ must satisfy $r_i = 2p-1$ for some $i\leq 3$. This implies $\theta_p(r_i)=-\infty$ for some $i\leq 3$, which in turn forces \eqref{eq:thetaq-leq-gap}.
\end{example}

In characteristic $p=2$, a smaller set of inequalities is sufficient to characterize WLP, as we explain next. For $a=\sum a_i\cdot 2^i$, $b=\sum b_i\cdot 2^i$, where $a_i,b_i\in\{0,1\}$ are the binary digits of $a,b$, we define the \defi{Nim-sum} $a\oplus b$ by performing digitwise addition modulo $2$:
\begin{equation}\label{eq:def-Nimsum} 
a\oplus b = c\text{ if and only if }a_i + b_i \equiv c_i \text{ (mod }2)\text{ for all }i,
\end{equation}
where $c_i$ are the binary digits of $c$.

\begin{theorem}\label{thm:WLP-char-2}
    If $\op{char}(\kk)=2$ then WLP holds for $M_{\ul{a}}$ if and only if
    \begin{equation}\label{eq:WLP-cond-char2} \sum_{i=1}^n \theta_q(a_i) \leq 2q-2 \text{ for all }q=2^k\quad\text{ with }\quad 2q>a_1\oplus a_2\oplus\cdots\oplus a_n.
    \end{equation}
    Moreover, in \eqref{eq:WLP-cond-char2} it suffices to consider those powers $q=2^k$ for which $q\leq 2\max(a_i)$.
\end{theorem}

For a concrete illustration of Theorem~\ref{thm:WLP-char-2}, see Example~\ref{ex:WLP-n=6}. Conditions \eqref{eq:WLP-cond-char2} arise as special cases of \eqref{eq:thetaq-leq-gap} when we consider $q$-reductions where $r_i$ is the remainder for the division of $a_i$ by $2q$. If  $2q>a_1\oplus\cdots\oplus a_n$ then these reductions are automatically balanced and even, so $g_+(\ul{a};\ul{r})=0$. This is unlike the case of odd primes $p$, where Example~\ref{ex:2p-1s} shows that odd balanced reductions cannot be ignored when applying Theorem~\ref{thm:new-WLP}. Since $\theta_q(r_i)\leq q-1$, the inequality \eqref{eq:thetaq-leq-gap} is automatic when $g_+(\ul{a};\ul{r})\geq n-2$, so we may always restrict to reductions that satisfy $g_+(\ul{a};\ul{r})\leq n-3$. In particular, for $n = 3$ it is enough to consider even balanced reductions. Example~\ref{ex:unbalanced-reductions-hook} shows that for $p>2$ it is necessary to consider unbalanced reductions in \eqref{eq:thetaq-leq-gap}, and moreover that the corresponding gap can be as large as $p-2$, which illustrates why the case $p=2$ is so special.

If we know that \eqref{eq:WLP-cond-char2} holds for some $q>\max_{i=1}^n(a_i)$ then it must also hold for all $\tilde{q}\geq q$, since the hypothesis implies $\theta_q(a_i)=a_i=\theta_{\tilde{q}}(a_i)$. This explains why the restriction $q\leq 2\max(a_i)$ is sufficient.

The quantity $s = a_1+\cdots+a_n$ is the \defi{socle degree of $M_{\ul{a}}$}, so $(M_{\ul{a}})_s$ is one-dimensional, and $(M_{\ul{a}})_e=0$ for $e>s$. If $s\leq 2p-2$ then no $q$ satisfies \eqref{eq:bounds-q-ineq-WLP}, so Theorem~\ref{thm:new-WLP} implies that WLP holds for all $M_{\ul{a}}$ with socle degree $s\leq 2p-2$. We prove more generally the following result, which in particular recovers \cite{Cook12}*{Proposition~3.5(ii)} and gives a new proof of \cite{Cook12}*{Conjecture~7.4} (see \cite{lund-nick}*{Theorem~4.4}, \cite{Vraciu}*{Theorem~1.6} for earlier proofs).

\begin{theorem}\label{thm:WLP-from-socle}
 Fix $n\geq 3$. Every monomial complete intersection $M_{\ul{a}}$ of socle degree $s$ satisfies WLP if and only if one of the following holds:
 \begin{enumerate}
     \item $s\leq 2p-2$.
     \item $n=3$, and $(2t+2)q-4 \leq s\leq (2t+2)q-2$ for some $t,q$ with $1\leq t<p$, $q=p^k\geq p$.
     \item $n=4$, $p=2$ and $s=6$.
 \end{enumerate}
\end{theorem}

To illustrate the difficulty in understanding WLP for monomial complete intersections, the next result shows that WLP for a given degree sequence $\ul{a}$ is essentially independent from the corresponding property for a subsequence. This is unlike the case of the strong Lefschetz property, where failure of SLP for a  subsequence of $\ul{a}$ implies the failure of SLP for $\ul{a}$. As was discussed before, if $a_{n+1}\geq a_1+\cdots+a_n$ then WLP holds for $M_{(a_1,\cdots,a_{n+1})}$, so the more difficult question is whether we can find $a_{n+1}$ to make WLP fail. 

\begin{theorem}\label{thm:WLP-fail}
    If $n\geq p\geq 3$ and $a_1,\cdots,a_n\geq 1$ then there exists $a_{n+1}$ such that WLP fails for $M_{(a_1,\cdots,a_{n+1})}$.
\end{theorem}

If $a_1+\cdots+a_n\leq p-1$ then $M_{(a_1,\cdots,a_{n+1})}$ satisfies WLP for all choices of $a_{n+1}$: this is because we either have $a_{n+1}\geq a_1+\cdots+a_n$  or $a_1+\cdots+a_{n+1}\leq 2p-2$, both of which imply WLP by the earlier discussion. It follows that the hypothesis $n\geq p$ of Theorem~\ref{thm:WLP-fail} is optimal. The conclusion of Theorem~\ref{thm:WLP-fail} also holds when $n>p=2$, but for $n=p=2$ there are exceptions, as discussed in Section~\ref{subsec:extend-seq}.

The key idea in our approach is to assemble all the monomial complete intersections into a single cohomological object, as follows. We let $R = \kk[x_1,\cdots,x_n,y_1,\cdots,y_n]$, and consider the local cohomology module $M=H^n_{(y_1,\cdots,y_n)}(R)$, which has a $\kk$-vector space decomposition
\[M=\bigoplus_{e_i,d_j\geq 0} \kk\cdot\frac{x_1^{e_1}\cdots x_n^{e_n}}{y_1^{1+d_1}\cdots y_n^{1+d_n}}.\]
There is a natural $\bb{Z}^n$-grading on $R$, where $\deg(x_i) = -\deg(y_i) = \vec{e}_i$ (the $i$-th standard unit vector), and we can give $M$ a $\bb{Z}^n$-graded $R$-module structure (which differs from the natural grading by $\ul{1}=(1,\cdots,1)$) by
\[M_{\ul{a}} = 0\text{ if some $a_i<0$, and otherwise }M_{\ul{a}} =  \bigoplus_{d_i+e_i=a_i}\kk\cdot\frac{x_1^{e_1}\cdots x_n^{e_n}}{y_1^{1+d_1}\cdots y_n^{1+d_n}}.\]
To reconcile this abuse of notation with \eqref{eq:def-Ma}, we write $T_i = x_iy_i$ and consider the polynomial subalgebra $\kk[T_1,\cdots,T_n]$ of $R$. It acts on $M$ by preserving the $\bb{Z}^n$-graded components, since $\deg(T_i)=\vec{0}$. If $a_i\geq 0$, each $M_{\ul{a}}$ is then a cyclic $\kk[T_1,\cdots,T_n]$-module, given~by
\begin{equation}\label{eq:Ma-cyclic-module}
M_{\ul{a}} = \kk[T_1,\ldots,T_n]\cdot\frac{1}{y_1^{1+a_1}\cdots y_n^{1+a_n}} \simeq \kk[T_1,\cdots,T_n]/\langle T_1^{1+a_1},\ldots,T_n^{1+a_n}\rangle.
\end{equation}
While \eqref{eq:def-Ma} defines $M_{\ul{a}}$ as an algebra, we are ultimately interested in understanding its structure as a $\kk[T]$-module, where $T=T_1+\cdots+T_n=x_1y_1+\cdots+x_ny_n$. The reason for this, as explained in \cite{lund-nick}*{Proposition~4.3}, is that the existence of a linear form $\ell$ for which \eqref{eq:timesT-maxrk} has maximal rank can be tested on $\ell=T$. Now the rank of multiplication by $T$ on the multigraded components of $M$ is controlled by the exact sequence
\[0 \lra H^{n-1}_{(y_1,\cdots,y_n)}(R/TR) \lra M \overset{\times T}{\lra} M \lra H^n_{(y_1,\cdots,y_n)}(R/TR) \lra 0. \]
The ring $R/TR$ has a geometric interpretation as the bihomogeneous coordinate ring of the incidence correspondence, the variety parametrizing pairs $(p,H)$ of a point in projective space and a hyperplane containing it. The local cohomology groups of $R/TR$ depend in intricate ways on the characteristic of $\kk$, and they have several equivalent reinterpretations (one of which will be used here, and is discussed in Section~\ref{subsec:coh-WLP}). A detailed study of the cohomology is undertaken in \cite{KMRR}, and in the current work we extract the crucial consequences that allow us to complete the characterization of WLP.

\medskip

\noindent{\bf Organization.} Section~\ref{sec:prelim} reviews the necessary background on characters and cohomology and translates WLP into a question about weights occurring in cohomology
characters. In Section~\ref{sec:bound-coh} we derive lower and upper bounds for the supports of cohomology characters. We use these bounds in Section~\ref{sec:characterize-WLP} to establish Theorem~\ref{thm:new-WLP}, and to derive the consequences presented above. In Section~\ref{sec:WLP-constant-degs} we explain how our results recover the classification of WLP for constant degree sequences from the work of Brenner--Kaid and Kustin--Vraciu. We conclude with a discussion of SLP in Section~\ref{sec:SLP}, employing the Green--Han--Monsky ring and the work of Renaud.

\section{Preliminaries}
\label{sec:prelim}

In this section, we review the necessary background on characters and representations, and use the gap and nonnegative gap functions \eqref{eq:defgap-g+} to characterize weights in a $2$-row Schur polynomial. Following \cite{KMRR}, we explain the relationship between monomial complete intersections and the characters of the cohomology groups of divided powers of the tautological subbundle on projective space, and we use this correspondence to reformulate WLP in cohomological terms. Finally, we explain how the function $\theta_q$ in \eqref{eq:def-theta-q} characterizes the weights that occur in products of two truncated complete symmetric polynomials.

\subsection{Conventions}
\label{subsec:conventions}
We work throughout over an arbitrary field $\kk$ of characteristic $p$. The reader may however assume that $\kk$ is infinite, or algebraically closed, since the cohomology groups we consider commute with extensions of the base field \cite{hartshorne}*{Proposition~III.9.3}, and the Lefschetz properties for monomial complete intersections only depend on the characteristic of the underlying field. In fact, as explained in \cite{MMN}*{Proposition~2.2} and \cite{lund-nick}*{Lemma~4.2,~Proposition~4.3}, the Lefschetz properties for $M_{\ul{a}}$ can be tested on the linear form $\ell = T = T_1+\cdots +T_n$ (this is formulated for WLP in loc. cit., but the argument is identical for SLP).  We will also assume throughout that $n\geq 2$ to avoid some degenerate situations, but the Lefschetz properties and the description of cohomology are trivial when $n=1$.

\subsection{Weights and characters}
\label{subsec:weight-chars}

Every representation $W$ of the algebraic torus $(\kk^\times)^n$ has a weight space decomposition \cite{BCRV}*{Lemma~9.7.9}
\[ W = \bigoplus_{\ul{u}=(u_1,\cdots,u_n)\in\bb{Z}^n} W_{\ul{u}}.\]
If $W$ is finite dimensional, we define the \defi{character of $W$} to be
\[[W] := \sum_{\ul{u}\in\bb{Z}^n} \dim(W_{\ul{u}})\cdot z_1^{u_1}\cdots z_n^{u_n} \in \bb{Z}[z_1^{\pm 1},\cdots,z_n^{\pm 1}].\]
We say that $\ul{u}$ is a weight in $W$ (or in $[W]$) if $W_{\ul{u}}\neq 0$. If $W$ is a representation of the algebraic group $\GL_n$, then $[W]$ is invariant under the coordinate permutation action of the symmetric group $\mf{S}_n$, that is, it belongs to the ring of symmetric Laurent polynomials (the \defi{character ring})
\begin{equation}\label{eq:def-Lambda} 
\Lambda = \bb{Z}[z_1^{\pm 1},\cdots,z_n^{\pm 1}]^{\mf{S}_n}.
\end{equation}
In particular, $[W]$ is determined by the \defi{dominant weights} $\ul{u}$, that is, those weights for which $u_1\geq u_2\geq\cdots\geq u_n$. 

If $\op{char}(\kk)=p$ then for each $q=p^k$, $k\geq 0$, we have the $q$-th Frobenius twist functor $F^q$. By a mild abuse of notation, we also write  $F^q$ for the corresponding endomorphisms of $\bb{Z}[z_1^{\pm 1},\cdots,z_n^{\pm 1}]$ and $\Lambda$, induced by $F^q(z_i)=z_i^q$. More generally, for polynomial extensions $\Lambda[t]$ or power series extensions $\Lambda[[X,Y]]$, we will write $F^q(t)=t^q$, $F^q(X)=X^q$, $F^q(Y)=Y^q$.

We introduce a partial order on $\bb{Z}[z_1^{\pm 1},\cdots,z_n^{\pm 1}]$ and $\Lambda$, defined by $f\geq g$ if $f-g$ has nonnegative coefficients. If $W,W'$ are representations of $(\kk^\times)^n$, we write $W\geq W'$ if $[W]\geq [W']$. In analogy with \eqref{eq:defgap-g+} we define
\begin{equation}\label{eq:gap-character}
g(W)=g([W])=\max\{g(\ul{u}) | W_{\ul{u}}\neq 0\},\quad g_+(W)=g_+([W])=\max\{g_+(\ul{u}) | W_{\ul{u}}\neq 0\},
\end{equation}
and note that if $W\geq W'$ then $g(W)\geq g(W')$ and $g_+(W)\geq g_+(W')$. If $W=0$ then we set $g(W)=-\infty$. 

Given a partial order $\geq$ on a ring $\Gamma$, we extend it to power series rings in one variable (and by induction in any number of variables) in the natural way: if $f(t)=\sum_{i\geq 0}f_i\cdot t^i$, $g(t)=\sum_{i\geq 0}g_i\cdot t^i$ are in $\Gamma[[t]]$ then $f(t)\geq g(t)$ if $f_i\geq g_i$ for all $i$. We will use this notation in Section~\ref{sec:bound-coh}
for $\Lambda[[t]]$ and $\Lambda[[X,Y]]$.

We will be particularly interested in the \defi{complete symmetric polynomials $h_d$}, $d\geq 0$, and the \defi{Schur polynomials $s_{(a,b)}$}, $a\geq b\geq 0$. For an introduction to Schur polynomials, the reader can consult \cite{fulton-YT}*{Section~2.2}. For a partition $\lambda$, let $\operatorname{SSYT}_n(\lambda)$ denote the set
of semistandard Young tableaux of shape $\lambda$ with entries in
$\{1,\dots,n\}$, that is, fillings of the Young diagram of $\lambda$ which are
weakly increasing along rows and strictly increasing down columns. For
$\mc{T}\in \operatorname{SSYT}_n(\lambda)$, the weight $\operatorname{wt}(\mc{T})=(u_1,\dots,u_n)$ records for each $k=1,\cdots,n$, the number $u_k$ of entries of $\mc{T}$ that are equal to $k$. If we write $z^{\mc{T}} = z_1^{u_1}\cdots z_n^{u_n}$ then
\[ s_{\ll} = \sum_{\mc{T}\in \operatorname{SSYT}_n(\lambda)} z^{\mc{T}}.\]

We have in particular
\begin{equation}
\begin{aligned}
    h_d &= s_{(d,0)} = \sum_{\substack{u_1+\cdots+u_n = d \\ u_j\geq 0}}z_1^{u_1}\cdots z_n^{u_n},\quad\text{ for }d\geq 0,\\
    s_{(a,b)} &= h_a\cdot h_b-h_{a+1}\cdot h_{b-1}=\sum_{\substack{
1\le i_1\le\cdots\le i_a\le n\\
1\le j_1\le\cdots\le j_b\le n\\
i_r<j_r,\ 1\le r\le b
}}
z_{i_1}\cdots z_{i_a}z_{j_1}\cdots z_{j_b},\quad\text{ for }a\geq b\geq 0.\\
\end{aligned}
\label{eq:def-hd-sab}
\end{equation}
More generally we set $h_d=0$ for $d<0$, and $s_{(a,b)} = h_a\cdot h_b-h_{a+1}\cdot h_{b-1}$ for all $a,b$. In particular, $s_{(d-1,d)}=0$ and $s_{(a,b)}=-s_{(b-1,a+1)}$.

\begin{lemma}\label{lem:wts-in-sab}
    Suppose that $a\geq b\geq 0$ and that $\ul{m}\in\bb{Z}^n_{\geq 0}$ with $|\ul{m}|=a+b$. The following are equivalent:
    \begin{enumerate}
        \item $\ul{m}$ is a weight in $s_{(a,b)}$.
        \item $\max(m_i)\leq a$.
        \item $g(\ul{m})\leq a-b$.
        \item $g_+(\ul{m})\leq a-b$.
    \end{enumerate}
\end{lemma}

\begin{proof}
\noindent $(1) \Leftrightarrow (2)$. If $\max(m_i)>a$ then any tableau of shape $(a,b)$ with $\op{wt}(\mc{T})=\ul{m}$ must have equal entries in some column. Such a tableau is not semistandard, hence $\ul{m}$ is not a weight in $s_{(a,b)}$. If instead $\max(m_i)\leq a$, then we can construct a semistandard tableau $\mc{T}$ of weight $\ul{m}$ by defining its first row $(i_1,\cdots,i_a)$ and second row $(j_1,\cdots,j_b)$ so that $(i_1,\cdots,i_a,j_1,\cdots,j_b)=(1^{m_1},\cdots,n^{m_n})$: to see that $\mc{T}$ is semistandard, note that if $i_r=j_r$ then $i_r=i_{r+1}=\cdots=i_a=j_1=\cdots=j_r$ have the same value $k$, which can only happen when $m_k>a$. It follows that $\ul{m}$ is a weight in~$s_{(a,b)}$ by \eqref{eq:def-hd-sab}.

\noindent $(2) \Leftrightarrow (3)$. Since $g(\ul{m})=2\max(m_i)-(a+b)$, we have $\max(m_i)\leq a$ if and only if $g(\ul{m})\leq 2a-(a+b)=a-b$.

\noindent $(3) \Leftrightarrow (4)$. Since $a-b\geq 0$ has the same parity as $|\ul{m}|=a+b$, it follows from \eqref{eq:defgap-g+} that $g(\ul{m})\leq a-b$ if and only if $g_+(\ul{m})\leq a-b$.
\end{proof}

\begin{corollary}\label{cor:g+m=mina-b}
    If $\ul{m}\in\bb{Z}^n_{\geq 0}$ then
\begin{equation}\label{eq:g+=mina-b}
g_+(\ul{m}) = \min\{a-b:\ul{m}\text{ is a weight in }s_{(a,b)}\}.
\end{equation}
    If $a\geq b\geq 0$ then (using notation \eqref{eq:gap-character}) we have
\begin{equation}\label{eq:gap-sab}
g(s_{(a,b)}) = g_+(s_{(a,b)}) = a-b\text{ for all }a\geq b\geq 0.
\end{equation}
\end{corollary}

\begin{proof} Lemma~\ref{lem:wts-in-sab} implies that if $\ul{m}$ is a weight in $s_{(a,b)}$ then $g_+(\ul{m})\leq a-b$, so in order to prove \eqref{eq:g+=mina-b} we need to find $a,b$ such that $\ul{m}$ is a weight in $s_{(a,b)}$ with $g_+(\ul{m})=a-b$. Since $|\ul{m}|$ and $g_+(\ul{m})$ have the same parity, we can then set
\[
a=\frac{|\ul{m}|+g_+(\ul{m})}{2}\quad\text{ and }\quad b=\frac{|\ul{m}|-g_+(\ul{m})}{2}.
\]

    Since $(a,b,0,\cdots,0)$ is a weight in $s_{(a,b)}$, it follows that $g_+(s_{(a,b)}) \geq g(s_{(a,b)})\geq a-b$. Using Lemma~\ref{lem:wts-in-sab} we get $g_+(s_{(a,b)})\leq a-b$, which proves \eqref{eq:gap-sab}.
\end{proof}

\subsection{Cohomology characters and WLP}
\label{subsec:coh-WLP}

We let $V=\kk^n$, let $S=\Sym(V) \simeq \kk[x_1,\cdots,x_n]$ be the symmetric algebra of $V$, and write $\PP=\bb{P}V=\op{Proj}(S)$ for the corresponding projective space parametrizing $1$-dimensional quotients of $V$. It comes with a tautological exact sequence
\begin{equation}\label{eq:ses-on-PV}
0 \lra \mc{R} \lra V \oo \mc{O}_{\PP} \lra \mc{O}_{\PP}(1) \lra 0,
\end{equation}
where $\mc{R}$ denotes the tautological rank $(n-1)$ subbundle of $V \oo \mc{O}_{\PP}$. For $d\geq 0$ we consider the divided power $D^d\mc{R}$ defined as the $\mf{S}_d$-invariant subsheaf of the tensor power $\mc{R}^{\oo d}$, relative to the permutation action of the symmetric group $\mf{S}_d$. We set $D^d\mc{R}=0$ for $d<0$. The sheaf cohomology groups of $D^d\mc{R}(e)$ are representations of $\GL_n$, and we consider their characters
\[ h^i(D^d\mc{R}(e)) = \left[H^i(\PP,D^d\mc{R}(e))\right].\]
It follows from \cite{KMRR}*{(1.24),\ (1.25),\ (2.7)} that
\begin{equation}\label{eq:identities-hiDdre}
    h^1(D^d\mc{R}(e)) = h^0(D^{e+1}\mc{R}(d-1)),\quad h^0(D^d\mc{R}(e))-h^1(D^d\mc{R}(e)) = s_{(e,d)},\quad\text{for }d\geq 0,e\geq -1.
\end{equation}

The algebra $M_{\ul{a}}$ is a complete intersection with socle degree $s=a_1+\cdots+a_n$, hence \eqref{eq:timesT-maxrk} has maximal rank for all $e$ if and only if it is surjective for $e\geq \lfloor s/2\rfloor$. Hence WLP fails for $M_{\ul{a}}$ if and only if
\[ \left(\frac{M_{\ul{a}}}{T\cdot M_{\ul{a}}}\right)_{e+1} \neq 0 \quad\text{ for some }e \geq \left\lfloor\frac{s}{2}\right\rfloor,\text{ if and only if}\quad\left(\frac{M_{\ul{a}}}{T\cdot M_{\ul{a}}}\right)_{e+1} \neq 0 \quad\text{ for }e = \left\lfloor\frac{s}{2}\right\rfloor,\]
where $T=T_1+\cdots+T_n$. As explained in \cite{KMRR}*{Section~2.5}, if $d+e=s$ then there is an identification
\[H^1(\PP,D^d\mc{R}(e))_{\ul{a}} = \left(\frac{M_{\ul{a}}}{T\cdot M_{\ul{a}}}\right)_{e+1}.\]
It follows that 
\begin{equation}\label{eq:coh-characterization-WLP}
\text{$M_{\ul{a}}$ fails WLP if and only if $\ul{a}$ is a weight in } h^1(D^d\mc{R}(e))\text{ for some }d,e\text{ with }d+e=s,\ e\geq d-1.
\end{equation}

Equivalently, the failure of WLP is characterized by the \defi{support sets}
\begin{equation}\label{eq:def-supp-d-e} \op{supp}(d,e) = \left\{ \ul{a}\in\bb{Z}^n : H^1\left(\PP^{n-1},D^d\mc{R}(e)\right)_{\ul{a}} \neq 0\right\}.
\end{equation}
Since the algebra $A=M_{\ul{a}}/(T\cdot M_{\ul{a}})$ is generated in degree $1$, it follows that for $e\geq 0$ the non-vanishing $A_{e+1}\neq 0$ implies $A_e\neq 0$, hence the support sets satisfy the inclusion relations
\begin{equation}\label{eq:incl-supp-de}
    \op{supp}(d,e) \supseteq \op{supp}(d-1,e+1) \supseteq \cdots \supseteq \op{supp}(d-f,e+f)\quad\text{ for all }f\geq 0.
\end{equation}
It is then natural to consider the characters
\begin{equation}\label{eq:def-gamma}
    \gamma_{d,e} = \sum_{f\geq 0}h^1(D^{d-f}\mc{R}(e+f)),
\end{equation}
and obtain from \eqref{eq:coh-characterization-WLP} that
\begin{equation}\label{eq:WLP-from-gamma}
\text{$M_{\ul{a}}$ fails WLP if and only if $\ul{a}$ is a weight in } \gamma_{d,e}\text{ for some }d,e\text{ with }d+e=s,\ e\geq d-1.
\end{equation}
    
We conclude this section by noting that the characterization \eqref{eq:coh-characterization-WLP} provides a simple explanation for the fact that WLP holds for all $\ul{a}$ when $n=2$ \cites{HMNW,MZ}: indeed, we have $D^d\mc{R}(e) \simeq \mc{O}_{\PP^1}(e-d)$, and the condition $e\geq d-1$ forces $H^1(\PP^1,\mc{O}_{\PP^1}(e-d))=0$.

\subsection{Weights in products of truncated symmetric polynomials}
\label{subsec:tensor-truncated}

In analogy with \eqref{eq:def-hd-sab} we define for each $q>0$ the \defi{truncated complete symmetric polynomials}
\begin{equation}\label{eq:def-trunc-hd}
h^{(q)}_d = \sum_{\substack{u_1+\cdots+u_n=d \\ 0\leq u_j<q}} z_1^{u_1}\cdots z_n^{u_n},
\end{equation}
noting that $h^{(q)}_d=0$ for $d<0$. We will be interested in the case when $q=p^k$ is a power of the characteristic, but the definition above and the following results make sense in general. We can interpret the function $\theta_q$ in \eqref{eq:def-theta-q} as measuring the gap for a truncated complete symmetric polynomial in $n=2$ variables:
\[\theta_q(r) = g\left(h^{(q)}_r(z_1,z_2)\right)\quad\text{ for }r=0,\cdots,2q-1.\]

\begin{lemma}\label{lem:thetaq-meaning}
    Consider integers $q\geq 1$ and $0\leq r\leq 2q-2$. For $\Delta\geq 0$, we can find $a,b$ satisfying
    \begin{equation}\label{eq:a-b-del-q} 
    r = a+b,\quad a-b \geq \Delta,\quad 0\leq a,b<q
    \end{equation}
    if and only if $\Delta \leq \theta_q(r)$.
\end{lemma}

\begin{proof}
    For the \emph{only if} direction, suppose that there exist $a,b\geq 0$ satisfying \eqref{eq:a-b-del-q}. If $r\leq q-1$ then 
    \[ \theta_q(r) = r = a+b \geq a-b \geq \Delta.\]
    If instead $q\leq r \leq 2q-2$ then $a\leq q-1$ and $b=r-a \geq r-(q-1)$, hence
    \[ \Delta \leq a-b \leq (q-1) - (r-(q-1)) = \theta_q(r).\]

    For the \emph{if} direction, suppose that $\Delta\leq\theta_q(r)$. If $r\leq q-1$ then we can take $a=r$, $b=0$, in which case we have $a-b=r=\theta_q(r)\geq\Delta$, and \eqref{eq:a-b-del-q} holds. If $q\leq r\leq 2q-2$ then $a=q-1$ and $b=r-(q-1)$ satisfy \eqref{eq:a-b-del-q}, concluding our proof.
\end{proof}

As a direct consequence of Lemma~\ref{lem:thetaq-meaning}, we get the following characterization of the terms that appear in a product of two truncated symmetric polynomials.

\begin{corollary}\label{cor:terms-in-hqu-hqv}
    Consider nonnegative integers $r_1,\cdots,r_n\leq 2q-2$. We have that $z_1^{r_1}\cdots z_n^{r_n}$ is a term in $h_u^{(q)}\cdot h_v^{(q)}$ for some $u,v$ with $u-v\geq\Delta$ if and only if
    \begin{equation}\label{eq:sum-thetari-geqDelta} 
    \sum_{i=1}^n \theta_q(r_i) \geq \Delta.
    \end{equation}
\end{corollary}

\begin{proof}
    Suppose first that $z_1^{r_1}\cdots z_n^{r_n}$ is a term in $h_u^{(q)}\cdot h_v^{(q)}$ for some $u,v$ with $u-v\geq\Delta$, and write $r_i=a_i+b_i$, with $\ul{a}$ a weight in $h_u^{(q)}$ and $\ul{b}$ a weight in $h_v^{(q)}$. If we set $\Delta_i = \max(a_i-b_i,0)$ then, since $\theta_q(r_i)\geq 0$, it follows from Lemma~\ref{lem:thetaq-meaning} that $\Delta_i\leq\theta_q(r_i)$. Combining this with the inequality $\Delta_i\geq a_i-b_i$, we conclude that
    \[\sum_{i=1}^n \theta_q(r_i) \geq \sum_{i=1}^n \Delta_i \geq \sum_{i=1}^n(a_i-b_i) = u-v \geq \Delta.\]

    Suppose next that \eqref{eq:sum-thetari-geqDelta} holds, and use Lemma~\ref{lem:thetaq-meaning} to find $0\leq a_i,b_i<q$ with $r_i=a_i+b_i$, $a_i-b_i = \theta_q(r_i)$. If we set $u=|\ul{a}|$ and $v=|\ul{b}|$ then $\ul{a}$ is a weight in $h_u^{(q)}$ and $\ul{b}$ is a weight in $h_v^{(q)}$. It follows that $z_1^{r_1}\cdots z_n^{r_n}$ is a term in $h_u^{(q)}\cdot h_v^{(q)}$ and
    \[u-v = \sum_{i=1}^n(a_i-b_i) = \sum_{i=1}^n \theta_q(r_i) \geq \Delta,\]
    concluding our proof.
\end{proof}

\section{Bounding cohomology characters}
\label{sec:bound-coh}

Although the complete description of the cohomology characters $h^1(D^d\mc{R}(e))$ is rather involved \cite{KMRR}, considerably less information is needed to characterize WLP via \eqref{eq:coh-characterization-WLP} or \eqref{eq:WLP-from-gamma}. The purpose of this section is to derive the bounds on these characters that will be used in the proof of the WLP criterion in Theorem~\ref{thm:new-WLP}. 

We start by considering the bivariate series $\bG(X,Y)\in\Lambda[[X,Y]]$, defined by (see \cite{KMRR}*{(1.7)})
\begin{equation}\label{eq:def-Guv}
    \bG(X,Y) = \sum_{d\geq 0,\ e\geq -1} h^1(D^d\mc{R}(e))\cdot X^d\cdot Y^{d+e},
\end{equation}
noting that for $d=0$, $e=-1$, we have $h^1(D^d\mc{R}(e))=0$, so all the exponents appearing in \eqref{eq:def-Guv} are indeed nonnegative. Recalling the conventions from Section~\ref{subsec:weight-chars}, we have that $\bG(X,Y)\geq 0$.

Motivated by \eqref{eq:WLP-from-gamma}, we define
\begin{equation}\label{eq:def-Wuv}
    \bW(X,Y) = \frac{1}{1-X}\cdot \bG(X,Y),
\end{equation}
and observe that for $d\geq 0$, $e\geq -1$, the coefficient of $X^d\cdot Y^{d+e}$ in $\bW(X,Y)$ is $\gamma_{d,e}$. 

If we define $\bh^{(q)}(t)\in\Lambda[t]$ for $q=p^k$ via
\[\bh^{(q)}(t) = \prod_{i=1}^n(1+tz_i+t^2z_i^2+\cdots+t^{q-1}z_i^{q-1}) = \sum_{d\geq 0}h^{(q)}_d(\ul{z})\cdot t^d,\]
then it follows from \eqref{eq:def-Wuv} and \cite{KMRR}*{Theorem~1.3} that 
\begin{equation}\label{eq:fun-eqn-WXY}
\bW(X,Y) = \bh^{(p)}(XY)\cdot \bh^{(p)}(Y)\cdot F^p(\bW(X,Y))+\frac{1}{1-X}\cdot\bN(X,Y),
\end{equation}
for some series $\bN(X,Y)\in\Lambda[[X,Y]]$. More precisely, for $c=1,\cdots,p-1$ there exist power series $\bN_c(t)\geq 0$ in $\Lambda[[t]]$ where all the exponents of $t$ have the same parity as $c+1$ such that (see \cite{KMRR}*{(5.19),\ (1.13)})
\[    
\bN(X,Y) = \sum_{c=1}^{p-1} \bN_c\left(X^{1/2}Y\right)\cdot X^{1/2}\cdot\frac{X^{c/2}-X^{p-c/2}}{1-X^p}.
\]
This shows that
\[\frac{1}{1-X}\cdot\bN(X,Y) = \sum_{c=1}^{p-1}\bN_c\left(X^{1/2}Y\right)\cdot X^{(c+1)/2}\cdot\frac{1+X+\cdots+ X^{p-1-c}}{1-X^p} \geq 0,\]
which combined with \eqref{eq:fun-eqn-WXY} implies
\begin{equation}\label{eq:W-geq-hhFW}
    \bW(X,Y) \geq \bh^{(p)}(XY)\cdot \bh^{(p)}(Y)\cdot F^p(\bW(X,Y)).
\end{equation}

\begin{lemma}\label{lem:bound-h1}
    If $q=p^k$, $d\geq 0$, $e\geq -1$, then we have a character inequality
    \[\gamma_{d,e} \geq \sum_{a\geq 0,\ b\geq -1}h^{(q)}_{d-aq}\cdot h^{(q)}_{e-bq}\cdot F^q\left(\gamma_{a,b}\right).\]
\end{lemma}

\begin{proof} Using the identity $\bh^{(pq)}(t)=\bh^{(q)}(t)\cdot F^q(\bh^{(p)}(t))$, it follows from \eqref{eq:W-geq-hhFW} by induction on $q$ that
\begin{equation}\label{eq:W-geq-hqhqFqW}
\bW(X,Y) \geq \bh^{(q)}(XY)\cdot \bh^{(q)}(Y)\cdot F^q(\bW(X,Y)).
\end{equation}
Writing $w_{d,e}$ for the coefficient of $X^d\cdot Y^{d+e}$ in $\bW(X,Y)$, it follows from \eqref{eq:def-Wuv} that $w_{d,e} \geq 0$, and as noted before $w_{d,e}=\gamma_{d,e}$ for $d\geq 0$ and $e\geq -1$. It follows from \eqref{eq:W-geq-hqhqFqW} that for $d\geq 0$, $e\geq -1$ we have
\[\gamma_{d,e}\geq \sum_{a,b}h^{(q)}_{d-aq}\cdot h^{(q)}_{e-bq}\cdot F^q\left(w_{a,b}\right) \geq \sum_{a\geq 0,\ b\geq -1}h^{(q)}_{d-aq}\cdot h^{(q)}_{e-bq}\cdot F^q\left(\gamma_{a,b}\right).    \qedhere\]
\end{proof}

To prove an upper bound for the support of cohomology characters we will use \cite{KMRR}*{Theorem~1.5}. For $q=p^k$ we define the
\defi{truncated Schur polynomials}
\begin{equation}\label{eq:def-trunc-sab} s^{(q)}_{(a,b)} = h^{(q)}_a\cdot h^{(q)}_b-h^{(q)}_{a+1}\cdot h^{(q)}_{b-1},
\end{equation}
and let
\begin{equation}\label{eq:def-Phi-de} 
\Phi_{d,e} = \sum_{j\geq 0} s^{(p)}_{(e+jp,d-jp)} \leq \sum_{j\geq 0} h^{(p)}_{e+jp}\cdot h^{(p)}_{d-jp}.
\end{equation}
We define
\begin{equation}\label{eq:def-W-de}
\mc{W}(d,e) = 
\left\{
 (q,a,b,u,v):
 \begin{array}{l}
 q=p^k,\ k\geq 1,\ a\geq0,\ -1\leq b\leq a-2,\
 u,v\geq 0\\ u\geq e-bq,\
 u+v+q(a+b)=d+e
 \end{array}
 \right\} 
\end{equation}
and observe that for $(q,a,b,u,v)\in\mc{W}(d,e)$ we must have $v\leq d-aq$.

\begin{lemma}\label{lem:upbd-H1DdRe}
    If $e\geq d-1$ then for every weight $\ul{w}$ in $h^1(D^d\mc{R}(e))$ there exists a tuple $(q,a,b,u,v)\in\mc{W}(d,e)$ such that $\ul{w}$ is a weight in $h^{(q)}_u\cdot h^{(q)}_v \cdot F^q\left(h^1(D^a\mc{R}(b))\right)$.
\end{lemma}

\begin{proof}
We prove the assertion by induction on $d+e$. We consider a weight $\ul{w}$ in $h^1(D^d\mc{R}(e))$, and in particular assume that $h^1(D^d\mc{R}(e))\neq 0$. If $d+e\leq 0$ then the condition $e\geq d-1$ forces either $d<0$ and $D^d\mc{R}=0$, or $d=0$, in which case $h^1(D^d\mc{R}(e))= 0$ for all $e\geq -1$. We may therefore assume that $d+e>0$.

Using \cite{KMRR}*{Theorem~1.5}, we have
\[ h^1(D^d\mc{R}(e)) = \sum_{A=0}^{\lfloor \frac{d}{p}\rfloor}\sum_{B=-1}^{\lfloor \frac{d+e}{p}\rfloor} \Phi_{d-Ap,e-Bp}\cdot F^p\left(h^1(D^A\mc{R}(B))\right),\]
and we note that \eqref{eq:def-Phi-de} implies 
\[\Phi_{d-Ap,e-Bp} \leq \sum_{\substack{D+E=d+e-p(A+B) \\ E\geq e-Bp}} h^{(p)}_E\cdot h^{(p)}_D.\]
It follows that
\begin{equation}\label{eq:1st-upbd-h1DdRe} h^1(D^d\mc{R}(e)) \leq \sum_{A=0}^{\lfloor \frac{d}{p}\rfloor}\sum_{B=-1}^{\lfloor \frac{d+e}{p}\rfloor}\sum_{\substack{D+E=d+e-p(A+B) \\ E\geq e-Bp}} h^{(p)}_E\cdot h^{(p)}_D \cdot F^p\left(h^1(D^A\mc{R}(B))\right),
\end{equation}
so $\ul{w}$ is a weight in one of the non-zero summands in \eqref{eq:1st-upbd-h1DdRe}. We analyze such summands according to the relation between $A$ and $B$. We note that they can only occur when $D,E\geq 0$, which combined with $d+e>0$ implies that $d+e>(d+e)/p\geq A+B$.

If $B\leq A-2$ then we let $q=p$, $a=A$, $b=B$, $u=E$ and $v=D$, to obtain a tuple $(q,a,b,u,v)\in\mc{W}(d,e)$ which satisfies the desired conclusion.

If $B\geq A-1$, then we can write \[\ul{w}=\ul{r}'+p\cdot\ul{w}',\text{ where $\ul{r}'$ is a weight in $h^{(p)}_E\cdot h^{(p)}_D$ and $\ul{w}'$ is a weight in $h^1(D^A\mc{R}(B))$.}\]
Since $A+B<d+e$, it follows by induction that we can find $(q_0,a,b,u_0,v_0)\in\mc{W}(A,B)$ such that $\ul{w}'$ is a weight in $h^{(q_0)}_{u_0}\cdot h^{(q_0)}_{v_0} \cdot F^{q_0}\left(h^1(D^a\mc{R}(b))\right)$. We can then decompose further
\[\ul{w}'=\ul{r}''+q_0\cdot\ul{w}'',\text{ where $\ul{r}''$ is a weight in $h^{(q_0)}_{u_0}\cdot h^{(q_0)}_{v_0}$ and $\ul{w}''$ is a weight in $h^1(D^a\mc{R}(b))$.}\]
We set $q=pq_0$, $u=pu_0+E$, $v=pv_0+D$, and note that $(q,a,b,u,v)\in\mc{W}(d,e)$ since $E\geq e-Bp$. Moreover,
\[ h^{(p)}_E\cdot F^p(h^{(q_0)}_{u_0}) \leq h^{(q)}_u,\quad h^{(p)}_D\cdot F^p(h^{(q_0)}_{v_0})\leq h^{(q)}_v,\]
which implies that $\ul{r}'+p\cdot\ul{r}''$ is a weight in $h^{(q)}_u\cdot h^{(q)}_v$. It follows that $\ul{w}=(\ul{r}'+p\cdot\ul{r}'')+q\cdot\ul{w}''$ is a weight in $h^{(q)}_u\cdot h^{(q)}_v \cdot F^q\left(h^1(D^a\mc{R}(b))\right)$, concluding the proof.
\end{proof}

\begin{proposition}\label{prop:bound-gap}
    Recalling notation \eqref{eq:gap-character}, we have for $e\geq d-1$ and $i=0,1$, that
    \[g\left(h^i(D^d\mc{R}(e))\right)\leq e-d.\]
\end{proposition}

\begin{proof}
We argue by induction on $d+e$. We have using \eqref{eq:identities-hiDdre} that 
\[h^0(D^d\mc{R}(e)) = s_{(e,d)} + h^1(D^d\mc{R}(e)),\]
and since $g(s_{(e,d)})\leq e-d$ by \eqref{eq:gap-sab} and the convention $g(s_{(d-1,d)})=g(0)=-\infty$, it suffices to prove the inductive step for $i=1$. Using Lemma~\ref{lem:upbd-H1DdRe}, it suffices to show that each non-zero character $h^{(q)}_u\cdot h^{(q)}_v \cdot F^q\left(h^1(D^a\mc{R}(b))\right)$ with $(q,a,b,u,v)\in\mc{W}(d,e)$ has gap at most $e-d$. As in the proof of Lemma~\ref{lem:upbd-H1DdRe}, we may assume that $d+e>0$, which in turn implies $d+e>a+b$. 

The condition $b\leq a-2$ implies
\[h^1(D^a\mc{R}(b))=h^0(D^{b+1}\mc{R}(a-1)).\]
Since $a-1\geq (b+1)-1$ and $(b+1)+(a-1)<d+e$, it follows by induction that $h^0(D^{b+1}\mc{R}(a-1))$ has gap at most $a-b-2$, and therefore
\[g\left(F^q\left(h^1(D^a\mc{R}(b))\right)\right) \leq q(a-b-2).\]

Suppose first that $v\geq q$. We have
\[ g\left(h^{(q)}_u\cdot h^{(q)}_v\right)\leq 2(2q-2)-(u+v)=(4q-4)-(d+e)+q(a+b).\]
It follows that the summand $h^{(q)}_u\cdot h^{(q)}_v \cdot F^q\left(h^1(D^a\mc{R}(b))\right)$ has gap bounded above by
\[q(a-b-2)+(4q-4)-(d+e)+q(a+b) = 2aq+2q-4-(d+e). \]
Recall that $d-aq\geq v\geq q$, hence $d\geq(a+1)q$. It follows that
\[2aq+2q-4-(d+e)\leq 2d-4-(d+e)=d-e-4 \leq e-d,\]
where the last inequality is equivalent to $e-d\geq -2$, which is true by assumption.

Suppose now that $v\leq q-1$. We have
\[ g\left(h^{(q)}_u\cdot h^{(q)}_v\right)\leq 2(q-1+v)-(u+v)=(2q-2)+2v-(d+e)+q(a+b).\]
It follows that the summand $h^{(q)}_u\cdot h^{(q)}_v \cdot F^q\left(h^1(D^a\mc{R}(b))\right)$ has gap bounded above by
\[q(a-b-2)+(2q-2)+2v-(d+e)+q(a+b)=2(aq+v)-2-(d+e)\leq 2d-2-(d+e)=d-e-2\leq e-d,\]
where the last inequality follows from $e-d\geq -1$, concluding our proof.
\end{proof}

\section{The characterization of WLP}
\label{sec:characterize-WLP}

The goal of this section is to prove Theorem~\ref{thm:new-WLP}, illustrate it with some concrete examples, and discuss the subsequent results highlighted in the Introduction.

\subsection{The proof of Theorem~\ref{thm:new-WLP}}
\label{subsec:proof-new-WLP}

We first explain why conditions~\eqref{eq:thetaq-leq-gap} are necessary. To that end, we suppose that \eqref{eq:thetaq-leq-gap} fails for some $q$-reduction $\ul{r}$ of $\ul{a}$, and our goal is to explain why WLP then fails for $\ul{a}$. If we write $g_+=g_+(\ul{a};\ul{r})$ then it follows from Corollary~\ref{cor:terms-in-hqu-hqv} that $\ul{r}$ is a weight in $h^{(q)}_u\cdot h^{(q)}_v$ for some $u,v$ satisfying
\[ u-v \geq (g_++2)q-1.\]
If we let $m_i = (a_i-r_i)/q$ so that $g_+=g_+(\ul{m})$ as in \eqref{eq:defgap-g+}, then Corollary~\ref{cor:g+m=mina-b} allows us to find $a,b$ with $a-1\geq b+1\geq 0$ such that $a-b=g_++2$ and $\ul{m}$ is a weight in $s_{(a-1,b+1)}$. It follows from \eqref{eq:identities-hiDdre} that
\[h^1(D^a\mc{R}(b)) = h^0(D^{b+1}\mc{R}(a-1)) = s_{(a-1,b+1)} + h^1(D^{b+1}\mc{R}(a-1)) \geq s_{(a-1,b+1)},\]
which then implies that $\ul{m}$ is a weight in $h^1(D^a\mc{R}(b))$.  

If we set $e = u+bq$, $d = v+aq$ then
\[e-d = (u-v) - (g_++2)q \geq -1,\text{ or equivalently }e\geq d-1.\]
Using $u=e-bq$, $v=d-aq$, this shows that
\[\ul{a} = \ul{r}+q\cdot\ul{m}\text{ is a weight in }h^{(q)}_{e-bq}\cdot h^{(q)}_{d-aq}\cdot F^q\left(h^1(D^a\mc{R}(b))\right).\]
Since $h^1(D^a\mc{R}(b))\leq \gamma_{a,b}$, Lemma~\ref{lem:bound-h1} implies that $\ul{a}$ is a weight in $\gamma_{d,e}$, hence WLP fails for $\ul{a}$ by \eqref{eq:WLP-from-gamma}.

To finish the proof of Theorem~\ref{thm:new-WLP}, we need to verify that conditions~\eqref{eq:thetaq-leq-gap} are also sufficient. We assume that WLP fails for $\ul{a}$ and proceed to find a reduction $\ul{r}$ for which \eqref{eq:thetaq-leq-gap} fails. By \eqref{eq:coh-characterization-WLP}, we can find $e\geq d-1$ such that $\ul{a}$ is a weight in $h^1(D^d\mc{R}(e))$. We recall the notation \eqref{eq:def-W-de} and apply Lemma~\ref{lem:upbd-H1DdRe} to find $(q,a,b,u,v)\in\mc{W}(d,e)$ such that $\ul{a}$ is a weight in $h^{(q)}_u\cdot h^{(q)}_v \cdot F^q\left(h^1(D^a\mc{R}(b))\right)$. It follows that we can write
 \[\ul{a} = \ul{r} + q\cdot\ul{m},\]
 where $\ul{r}$ is a weight in $h^{(q)}_u\cdot h^{(q)}_v$ and $\ul{m}$ is a weight in $h^1(D^a\mc{R}(b))=h^0(D^{b+1}\mc{R}(a-1))$. By construction we have $a_i\equiv r_i\ (\op{mod}\ q)$ and $r_i\leq a_i$, and the fact that $\ul{r}$ is a weight in $h^{(q)}_u\cdot h^{(q)}_v$ implies $0\leq r_i\leq 2q-2$, hence $\ul{r}$ is a $q$-reduction of $\ul{a}$.
 
 Since $\ul{m}$ is a weight in $h^0(D^{b+1}\mc{R}(a-1))$, it follows from Proposition~\ref{prop:bound-gap} that $\ul{m}$ satisfies $g(\ul{m})\leq a-b-2$. Since $|\ul{m}|=a+b$ has the same parity as $a-b-2$, we get $g_+(\ul{a};\ul{r})=g_+(\ul{m})\leq a-b-2$. It follows from Corollary~\ref{cor:terms-in-hqu-hqv} that 
 \[\sum_{i=1}^n\theta_q(r_i)\geq u-v \geq (e-d)+(a-b)q \geq -1+(g_+(\ul{a};\ul{r})+2)q,\]
 proving that \eqref{eq:thetaq-leq-gap} fails, as desired.

The fact that it suffices to consider only those $q$ that satisfy \eqref{eq:bounds-q-ineq-WLP} was already explained in the introduction.

\subsection{Socle degrees that guarantee the weak Lefschetz property}\label{subsec:WLP-from-socle} Our next goal is to verify Theorem~\ref{thm:WLP-from-socle}, which we do by taking advantage of the main cohomology vanishing result from \cite{gao-raicu}. 

\begin{proof}[Proof of Theorem~\ref{thm:WLP-from-socle}]
 We fix a socle degree $s$ and note that Theorem~\ref{thm:new-WLP} implies that WLP holds when $s\leq 2p-2$, because $p>(s+1)/2$, hence there is no prime power $q$ satisfying the inequality \eqref{eq:bounds-q-ineq-WLP}. We may therefore assume that $s\geq 2p-1$, we set $d=e=s/2$ if $s$ is even, and $d=(s+1)/2$, $e=(s-1)/2$ if $s$ is odd. By the discussion in Section~\ref{subsec:coh-WLP}, the assertion that WLP holds for all algebras $M_{\ul{a}}$ with $a_1+\cdots+a_n=s$ is equivalent to the cohomology vanishing
 \begin{equation}\label{eq:H1vanishing-smalld} H^1\left(\PP^{n-1},D^d\mc{R}(e)\right) = 0.\end{equation}
 The assumption $s\geq 2p-1$ implies $d\geq p$, and therefore we can find unique integers $t$ and $q=p^k$ with
 \[ 1\leq t< p\leq q,\quad tq\leq d < (t+1)q.\]
 Using \cite{gao-raicu}*{Theorem~1.3 and (3.14)}, the vanishing \eqref{eq:H1vanishing-smalld} is equivalent to
 \begin{equation}\label{eq:H1van-bound-e} e \geq (t+n-2)q-n+1.\end{equation}
 
 Suppose first that $s$ is even, so $e=d\leq(t+1)q-1$, in which case \eqref{eq:H1van-bound-e} implies $n-2\geq(n-3)q\geq 2(n-3)$, which can only hold if $n=3,4$. If $n=3$ then we get $(t+1)q-2\leq e\leq (t+1)q-1$, hence $s=d+e$ implies
 \[ s = (2t+2)q-4 \quad \text{ or }\quad s = (2t+2)q-2.\]
 If $n=4$ then the inequalities $(t+1)q-1\geq e\geq (t+2)q-3$ hold if and only if $q=2$, which in turn forces $p=2$, $t=1$, and $e=3$, hence $s=6$. This shows that conditions (2), (3) in Theorem~\ref{thm:WLP-from-socle} are necessary when $s$ is even. If instead $s$ is odd, then $e=d-1\leq(t+1)q-2$. Condition \eqref{eq:H1van-bound-e} implies $n-3\geq (n-3)q$ which can only hold when $n=3$, which then forces $e=(t+1)q-2$, and $s=2e+1=(2t+2)q-3$, proving that (2), (3) are necessary also when $s$ is odd. It is now clear that in both cases (2), (3), the condition \eqref{eq:H1van-bound-e} is satisfied, which shows that conditions (2), (3) are sufficient and concludes our proof.
\end{proof}

\subsection{Some special cases of WLP}
\label{subsec:elementary-WLP}

We next give some concrete illustrations of Theorem~\ref{thm:new-WLP}. The first one explains why WLP holds in the unbalanced case (see \cite{mig-mr}*{Proposition~5.2}, \cite{lund-nick}*{Theorem~2.5}).

\begin{lemma}\label{lem:gap-conditions-unbalanced}
    If $a_1\geq a_2+\cdots+a_n$ then conditions \eqref{eq:thetaq-leq-gap} hold for all $q$.
\end{lemma}

\begin{proof}
Consider a $q$-reduction $\ul{r}$ of $\ul{a}$ and let $\ul{m}$ such that $a_i = qm_i + r_i$ for all $i$. We have
\[ m_1 - (m_2+\cdots+m_n) \leq g(\ul{m}) \leq g_+(\ul{m})\]
so the hypothesis implies
\begin{equation}\label{eq:bound-g+mq}
\begin{aligned}
0 &\leq a_1-a_2-\cdots-a_n = (m_1-m_2-\cdots-m_n)q+(r_1-r_2-\cdots-r_n) \\
&\leq g_+(\ul{m})q+(r_1-r_2-\cdots-r_n).
\end{aligned}
\end{equation}
Notice that for all $x\geq 0$ we have $\theta_q(x)\leq x$, and moreover if $0\leq x\leq 2q-1$ then
\[ x + \theta_q(x) \leq 2q-2.\]
Indeed, if $x\leq q-1$ then $x + \theta_q(x)=2x\leq 2q-2$, if $q\leq x\leq 2q-2$ then $x+\theta_q(x)=2q-2$, while for $x=2q-1$ we have $x+\theta_q(x)=-\infty$. We obtain
\[\sum_{i=1}^n\theta_q(r_i) \leq (2q-2-r_1) + (r_2+\cdots+r_n) \leq 2q-2 + g_+(\ul{m})q,\]
where the last inequality follows from \eqref{eq:bound-g+mq}. Since $g_+(\ul{m})=g_+(\ul{a};\ul{r})$, this proves \eqref{eq:thetaq-leq-gap}, as desired.
\end{proof}

\begin{lemma}\label{lem:hooks}
    If $\ul{a}=(a,1^{n-1})$ then $M_{\ul{a}}$ has WLP if and only if one of the following holds:
    \begin{enumerate}
        \item $a\geq n-1$, or
        \item $a\leq 2p-1-n$.
    \end{enumerate}
    In particular,  $M_{\ul{a}}$ has WLP for all $n\leq p$.
\end{lemma}
\begin{proof} If $a\geq n-1$ then $M_{\ul{a}}$ has WLP by Lemma~\ref{lem:gap-conditions-unbalanced}, while for $a\leq 2p-1-n$ WLP holds by Theorem~\ref{thm:WLP-from-socle}(1), since the socle degree of $M_{\ul{a}}$ is $s=a+n-1\leq 2p-2$. It remains to show that WLP fails for $2p-n\leq a\leq n-2$.

If $a<p$ then $\ul{r}=\ul{a}$ is a $p$-reduction with $g_+(\ul{a};\ul{r})=0$ and
\[\sum_{i=1}^n\theta_p(r_i)=a+n-1>2p-2,\]
so \eqref{eq:thetaq-leq-gap} fails for $\ul{r}$. If instead $a\geq p$ then we can write $a=(g+1)p+r$ with $0\leq r\leq p-1$. For $r<p-1$ we consider the $p$-reduction $\ul{r}=(p+r,1^{n-1})$: it satisfies $g_+(\ul{a};\ul{r})=g$, and using that $a\leq n-2$ we obtain
\[\sum_{i=1}^n\theta_p(r_i)=2p-2-(p+r)+n-1\geq p-1-r+a=(g+2)p-1>(g+2)p-2,\]
so \eqref{eq:thetaq-leq-gap} fails for $\ul{r}$. If $r=p-1$ then we consider the $p$-reduction $\ul{r}=(r,1^{n-1})$: it satisfies $g_+(\ul{a};\ul{r})=g+1$ and
\[\sum_{i=1}^n\theta_p(r_i)=p-1+n-1\geq p+a=(g+3)p-1>(g+3)p-2.\] 
This shows that \eqref{eq:thetaq-leq-gap} fails for $\ul{r}$ and concludes our proof. 
\end{proof}

\begin{example}\label{ex:unbalanced-reductions-hook}
Suppose that $\ul{a} = ((g+1)p-1,1^{n-1})$ for $0\leq g\leq p-2$. We have by Lemma~\ref{lem:hooks} that WLP holds for $M_{\ul{a}}$ if and only if $n\leq(g+1)p$. To see directly how Theorem~\ref{thm:new-WLP} applies, we note that every $p$-reduction $\ul{r}$ of $\ul{a}$ must satisfy $r_i=a_i=1$ for $i\geq 2$, and $r_1\in\{p-1,2p-1\}$. Since $\theta_p(2p-1)=-\infty$, the only $p$-reduction for which \eqref{eq:thetaq-leq-gap} is non-trivial is $\ul{r}=(p-1,1^{n-1})$, and it satisfies $g_+(\ul{a};\ul{r}) = g$. It follows from \eqref{eq:thetaq-leq-gap} that
\[(p-1)+(n-1) = \sum_{i=1}^n\theta_p(r_i) \leq (g+2)p-2,\]
which is equivalent to $n\leq (g+1)p$. When $(g+1)p<n\leq p(2p-1-g)$, condition \eqref{eq:thetaq-leq-gap} holds for all $q\geq p^2$, so the unique $q$-reduction that exhibits the failure of WLP arises for $q=p$ and $\ul{r}=(p-1,1^{n-1})$. In particular, it is necessary to consider positive gaps $g_+(\ul{a};\ul{r})$ in \eqref{eq:thetaq-leq-gap}, which must be as large as $p-2$. Compare this with the next section, where we will show that for $p=2$ positive gaps may be ignored when characterizing~WLP.
\end{example}

\subsection{Weak Lefschetz property in characteristic $2$}\label{subsec:WLP-char2} 
We now specialize to the case when $\kk$ has characteristic $p=2$. As usual, $q$ denotes a power of the characteristic. We begin by recalling \cite{KMRR}*{Theorem~1.4}, which employs the \defi{Nim symmetric polynomials} $\mc{N}_m=\mc{N}_m(\ul{z})$, defined for $m\geq 0$ (using notation \eqref{eq:def-Nimsum}) by
\begin{equation}\label{eq:def-Nim-pol}
 \mc{N}_m = \sum_{\substack{i_1+\cdots+i_n=2m \\ i_1\oplus i_2\oplus\cdots\oplus i_n = 0}} z_1^{i_1} z_2^{i_2}\cdots z_n^{i_n}.
\end{equation}
If $d\geq 0$ and $e\geq d-1$, then with notation \eqref{eq:def-trunc-sab} and \eqref{eq:def-Nim-pol}, \cite{KMRR}*{Theorem~1.4} yields
\begin{equation}\label{eq:h1DdRe-char2}  h^1(D^d\mc{R}(e)) = \sum_{\substack{q=2^k\geq 2 \\ m,j\geq 0}} F^{2q}(\mc{N}_m)\cdot s^{(q)}_{(e-(2m-2j-1)q,d-(2m+2j+1)q)}.
\end{equation}

\begin{proof}[Proof of Theorem~\ref{thm:WLP-char-2}]
    Using the notation \eqref{eq:def-trunc-hd}, we have the identity
    \begin{equation}\label{eq:telescope-qschur} \sum_{j\geq 0} s^{(q)}_{(u+j,v-j)} = \sum_{j\geq 0} \left(h^{(q)}_{u+j}\cdot h^{(q)}_{v-j} - h^{(q)}_{u+j+1}\cdot h^{(q)}_{v-j-1}\right) = h^{(q)}_u\cdot h^{(q)}_v.
    \end{equation}
Recalling notation \eqref{eq:def-gamma}, it follows from \eqref{eq:telescope-qschur} and \eqref{eq:h1DdRe-char2} that for $e\geq d-1$ we have
\[ \gamma_{d,e} = \sum_{\substack{q=2^k\geq 2 \\ m,j\geq 0}} F^{2q}(\mc{N}_m)\cdot h^{(q)}_{e-(2m-2j-1)q}\cdot h^{(q)}_{d-(2m+2j+1)q}.\]
Writing $u=e-(2m-2j-1)q$, $v=d-(2m+2j+1)q$, we note that $e\geq d-1$ implies 
\[ u-v = e - d + (4j+2)q \geq 2q-1.\]
It follows from \eqref{eq:WLP-from-gamma} that $M_{\ul{a}}$ fails WLP if and only if
\[ \ul{a}\text{ is a weight in }F^{2q}(\mc{N}_m)\cdot h_u^{(q)}\cdot h_v^{(q)}\text{ for some }q=2^k\geq 2,\ u,v\geq 0,\text{ with }u-v\geq 2q-1.\]

Consider now a tuple $\ul{a}$ for which the condition above holds, and define $b_i,r_i$ such that
\begin{equation}\label{eq:def-bi-ri} a_i = 2qb_i + r_i,\ 0\leq r_i <2q\quad\text{ for }i=1,\cdots,n.
\end{equation}
It follows that $\ul{b}$ is a weight in $\mc{N}_m$, and $\ul{r}$ is a weight in $h_u^{(q)}\cdot h_v^{(q)}$. Since the Nim sum of $b_1,\cdots,b_n$ is zero, and the powers of $2$ that appear in the $2$-adic expansion of each $r_i$ are bounded by $q$, we get
\[ a_1\oplus \cdots \oplus a_n = r_1 \oplus \cdots \oplus r_n \leq 1 + 2 + 4  + \cdots + q = 2q-1 < 2q.\]
Moreover, applying Corollary~\ref{cor:terms-in-hqu-hqv} with $\Delta=2q-1$, we get
\[ \sum_{i=1}^n \theta_q(a_i) = \sum_{i=1}^n \theta_q(r_i) \geq 2q-1.\]
It follows that if $M_{\ul{a}}$ fails WLP, then condition \eqref{eq:WLP-cond-char2} must fail for some $q=2^k$.

Conversely, suppose that \eqref{eq:WLP-cond-char2} fails for some $q=2^k$, and define $b_i,r_i$ via \eqref{eq:def-bi-ri}. We have as before that $r_1\oplus\cdots\oplus r_n < 2q$, hence
\[ 2q > a_1 \oplus \cdots \oplus a_n = 2q(b_1\oplus\cdots\oplus b_n) + (r_1\oplus\cdots\oplus r_n),\]
which forces $b_1\oplus\cdots\oplus b_n=0$. Assuming without loss of generality that $b_1=\max(b_i)$, we get $b_1=b_2\oplus\cdots\oplus b_n\leq b_2+\cdots+b_n$, so $g(\ul{b})\leq 0$. It follows that $\ul{r}$ is an even balanced $q$-reduction of $\ul{a}$, hence $g_+(\ul{a};\ul{r}) = 0$. The failure of \eqref{eq:WLP-cond-char2} translates into
\[ 2q - 1 \leq \sum_{i=1}^n \theta_q(a_i) = \sum_{i=1}^n \theta_q(r_i), \]
which implies that \eqref{eq:thetaq-leq-gap} is not satisfied, and therefore $M_{\ul{a}}$ fails WLP by Theorem~\ref{thm:new-WLP}.
\end{proof}

\begin{example}\label{ex:WLP-n=6}
    Suppose $n=6$ and $\ul{a}=(10, 3, 3, 2, 1, 1)$. We have $\max(a_i)=10$ and
    \[ a_1\oplus a_2\oplus\cdots\oplus a_n = 8,\]
    so it suffices to consider $q=8$ and $q=16$. We have $\theta_8(10) = 4$ and in all other cases $\theta_q(a_i)=a_i$, hence
    \[ \sum_{i=1}^n \theta_q(a_i) = \begin{cases}
        14 & \text{if }q = 8,\\
        20 & \text{if }q=16,
    \end{cases}\]
    and in both cases the inequality \eqref{eq:WLP-cond-char2} holds. It follows that
    $\kk[T_1,T_2,T_3,T_4,T_5,T_6]/\langle T_1^{11}, T_2^4, T_3^4, T_4^3, T_5^2, T_6^2\rangle$
    satisfies WLP (since $\ul{a}$ is unbalanced, the same conclusion follows from Lemma~\ref{lem:gap-conditions-unbalanced}). 
    
    If we consider instead $\ul{a}=(9, 3, 2, 2, 2, 1)$ then $\max(a_i)=9$ and 
    \[ a_1\oplus a_2\oplus\cdots\oplus a_n = 9,\]
    so we again have to consider $q=8$ and $q=16$. We have
    \[ \sum_{i=1}^n \theta_q(a_i) = 15 > 2q-2 \quad\text{ for }q=8,\]
    hence WLP fails for 
    $\kk[T_1,T_2,T_3,T_4,T_5,T_6]/\langle T_1^{10}, T_2^4, T_3^3, T_4^3, T_5^3, T_6^2\rangle$.
\end{example}

\subsection{WLP under extending a degree sequence}
\label{subsec:extend-seq}

Lemma~\ref{lem:gap-conditions-unbalanced} shows that given any degree sequence $\ul{a}$ we can find $a_{n+1}$ such that the extended sequence $(a_1,\cdots,a_{n+1})$ satisfies WLP (any sufficiently large $a_{n+1}$ works). We show that for most sequences $\ul{a}$ we can also find $a_{n+1}$ such that the extended sequence fails WLP.

\begin{lemma}\label{lem:large-thetap}
    If $x\geq p\geq 3$ then 
    \[\max\left(\theta_p(x),\theta_p(x-p)\right) \geq \frac{p-1}{2}.\]
\end{lemma}

\begin{proof} We write $x=2pm+r$ with $0\leq r<2p$. If $r=p-1$ then $\theta_p(x)=p-1$ and if $r=2p-1$ then $\theta_p(x-p)=p-1$, so the desired inequality holds in both cases. Otherwise, we have
\[\theta_p(x)+\theta_p(x-p) = p-2 = \frac{p-1}{2}+\frac{p-3}{2},\]
which implies the desired conclusion.    
\end{proof}

\begin{proof}[Proof of Theorem~\ref{thm:WLP-fail}] Without loss of generality, we may assume that for some $0\leq n'\leq n$ we have
\[ a_1,\cdots,a_{n'}\geq p,\quad\text{and}\quad p-1\geq a_{n'+1},\cdots,a_n\geq 1.\]
Using Lemma~\ref{lem:large-thetap}, we can find a $p$-reduction $\ul{r}$ of $\ul{a}$ with $\theta_p(r_i)\geq(p-1)/2$ for $i=1,\cdots,n'$. Since $p-1\geq r_i=a_i\geq 1$ for $i>n'$, it follows that
\[\sum_{i=1}^n\theta_p(r_i) \geq n'\cdot\frac{p-1}{2} + (n-n') = n + n'\cdot\frac{p-3}{2}\geq n\geq p.\]
We set $m_i=(a_i-r_i)/p$ for $i=1,\cdots,n$ and define
\[ m_{n+1} = m_1+\cdots+m_n\quad\text{and}\quad a_{n+1} = p\cdot m_{n+1}+(p-1).\]
If we set $r_{n+1}=p-1$ then we obtain a $p$-reduction $(r_1,\cdots,r_{n+1})$ of $(a_1,\cdots,a_{n+1})$ with
\[\sum_{i=1}^{n+1}\theta_p(r_i) \geq 2p-1.\]
Since by construction we have $g_+(m_1,\cdots,m_{n+1})=0$, this shows that conditions \eqref{eq:thetaq-leq-gap} are not satisfied for $(a_1,\cdots,a_{n+1})$ and therefore WLP fails for this degree sequence.
\end{proof}

We next explain why Theorem~\ref{thm:WLP-fail} also holds when $n\geq 3$ and $p=2$. For $q=2^k$, we get as in the proof of Lemma~\ref{lem:large-thetap} that for $x\geq q$
\[\max\left(\theta_q(x),\theta_q(x-q)\right) \geq \frac{q-2}{2},\text{ with equality if }x\in\left\{\frac{q-2}{2},\frac{3q-2}{2}\right\}\text{ (mod $2q$)}\]
We can then take $q=4$ and assume as in the proof of Theorem~\ref{thm:WLP-fail} that for some $0\leq n'\leq n$ we have
\[ a_1,\cdots,a_{n'}\geq 4,\quad\text{and}\quad 3\geq a_{n'+1},\cdots,a_n\geq 1.\]
We choose a reduction $\ul{r}$ with $\theta_4(r_i)\geq 1$ for $i=1,\cdots,n'$, $r_i=a_i$ for $i>n'$, and obtain
\begin{equation}\label{eq:4reds-geqn}
\sum_{i=1}^n\theta_4(r_i) \geq n' + (n-n') = n.
\end{equation}
If $n\geq 4$ or if $n=3$ and the above inequality is strict, then $\sum_{i=1}^n\theta_4(r_i)\geq 4$. We set $m_i=(a_i-r_i)/4$ for $i=1,\cdots,n$, let $m_{n+1}=m_1+\cdots+m_n$, $a_{n+1}=4m_{n+1}+3$ and choose $r_{n+1}=3$ to obtain $\sum_{i=1}^{n+1}\theta_4(r_i) \geq 7$, so the $4$-reduction $(r_1,\cdots,r_{n+1})$ of $(a_1,\cdots,a_{n+1})$ violates \eqref{eq:thetaq-leq-gap}. 

We are left to consider the case when $n=3$ and \eqref{eq:4reds-geqn} is an equality, which forces $a_i\in\{1,5\}\text{ (mod $8$)}$, and therefore $\theta_2(a_i)=1$ for all $i=1,2,3$. We set $r_i=1$ and $m_i=(a_i-r_i)/2$ for $i=1,\cdots,3$, so $\ul{r}=(r_1,r_2,r_3)$ is a $2$-reduction. We let $m_{4}=m_1+m_2+m_3$, $a_{4}=2m_{4}+1$ and choose $r_{4}=1$. We obtain an even balanced $2$-reduction $(r_1,r_2,r_3,r_4)$ of $(a_1,a_2,a_3,a_4)$, and
$\sum_{i=1}^{4}\theta_2(r_i) = 4 > 2=2p-2$,
which violates \eqref{eq:thetaq-leq-gap}.

\begin{example}\label{ex:pn=2-allWLP}
    To see why Theorem~\ref{thm:WLP-fail} fails for $n=p=2$, take $(a_1,a_2)=(1,2)$. We argue that WLP holds for $(1,2,a_3)$ for all $a_3\geq 0$. When $a_3\geq 3=a_1+a_2$, this follows from Lemma~\ref{lem:gap-conditions-unbalanced}. When $a_3\in\{1,2\}$, $M_{(a_1,a_2,a_3)}$ has socle degree $s\in\{4,5\}$, which satisfies the inequality in Theorem~\ref{thm:WLP-from-socle}(2) for $t=1$ and $q=2$, which implies WLP. Finally, if $a_3=0$ then $M_{(a_1,a_2,a_3)}=M_{(a_1,a_2)}$ and, as noted before, WLP always holds in $2$ variables.
\end{example}

\section{Classification of WLP for constant degree sequences}
\label{sec:WLP-constant-degs}

The goal of this section is to explain how our main results recover the classification of WLP from \cite{bre-kai} and 
\cite{kus-vra}, for the monomial complete intersections
\begin{equation}\label{eq:KT-mod-Td}
\kk[T_1,\ldots,T_n]/\langle T_1^{d},\ldots,T_n^{d}\rangle,\quad \text{with }d\geq 2.
\end{equation}
With notation \eqref{eq:def-Ma}, these are the algebras $M_{\ul{a}}$ where
\begin{equation}\label{eq:ai=d-1} 
a_1=\cdots=a_n=d-1,
\end{equation}
which we assume throughout this section. Since WLP holds for $n\leq 2$, we will also assume that $n\geq 3$.

\subsection{Possible $q$-reductions and a characterization of WLP}

We assume that $q=p^k\geq p$ and consider the possible $q$-reductions for the degree sequence $\ul{a}$ in \eqref{eq:ai=d-1}. If $d\leq q$ then there exists a unique $q$-reduction $\ul{r}$ of~$\ul{a}$, given by
\[ r_i = d-1\quad\text{ for }i=1,\cdots,n.\]
It satisfies $g_+(\ul{a};\ul{r})=0$ and we have the equivalence
\[\sum_{i=1}^n \theta_q(r_i) = n(d-1)\leq 2q-2 \quad\Longleftrightarrow\quad d\leq\frac{2q+n-2}{n}.\]
It follows from Theorem~\ref{thm:new-WLP} that
\begin{equation}\label{eq:Md-easyfail-WLP}
M_{\ul{a}}\text{ fails WLP if }\quad\frac{2q+n-1}{n}\leq d\leq q.
\end{equation}

If $d-1 \geq q$ then we write
\begin{equation}\label{eq:d-1=2mq+dbar} 
d-1 = 2mq+\ol{d},\quad\text{where }0\leq \ol{d}\leq 2q-1,
\end{equation}
and note that for each $q$-reduction $\ul{r}$ of $\ul{a}$ and each $i=1,\cdots,n$, we either have
\begin{equation}\label{eq:possible-aibar}
    r_i = \ol{d} \quad\text{or}\quad r_i = \begin{cases}
        \ol{d}+q & \text{if }0\leq\ol{d}\leq q-1, \\
        \ol{d}-q & \text{if }q\leq\ol{d}\leq 2q-1. \\
    \end{cases}
\end{equation}

Since $\theta_q(2q-1)=-\infty$, it follows that if $\ol{d}\in\{q-1,2q-1\}$ (or equivalently, if $q|d$) then there is only one choice of $q$-reduction that makes \eqref{eq:thetaq-leq-gap} non-trivial, namely
\begin{equation}\label{eq:qred-allq-1}
    r_1=\cdots=r_n=q-1.
\end{equation}
Writing $a_i = m_iq+r_i$, we get $m_i=2m$ when $\ol{d}=q-1$ and $m_i = 2m+1$ when $\ol{d}=2q-1$. In particular \eqref{eq:qred-allq-1} is a balanced reduction, and it is odd if and only if $n$ is odd and $\ol{d}=2q-1$. Equation \eqref{eq:thetaq-leq-gap} becomes
\[n(q-1)\leq 2q-2\quad\text{if $n$ is even or $\ol{d}=q-1$},\]
which contradicts our assumption $n\geq 3$, and
\[n(q-1)\leq 3q-2\quad\text{if $n$ is odd and $\ol{d}=2q-1$},\]
which holds only for $n=3$. It follows that (see also \cite{bre-kai}*{Corollary~2.9})
\begin{equation}\label{eq:WLP-p-divides-d}
\text{if $p|d$ then WLP may only hold when $n=3$ and $d=(2m+2)p$ is an even multiple of $p$}.
\end{equation}

If instead $\ol{d}\not\in\{q-1,2q-1\}$, then we have
\begin{equation}\label{eq:theta-dbar-pm-q}
\theta_q(\ol{d}+q) = \theta_q(\ol{d}-q) = q-2-\theta_q(\ol{d}).
\end{equation}

\begin{theorem}\label{thm:WLP-constant-ai}
    Fix $d\geq 2$ and for $q=p^k\geq p$ let
    \[\xi_q = \begin{cases}
        \theta_q(d-1)=d-1 & \text{if }d\leq q \\
        \max\{\theta_q(d-1),\theta_q(d-1-q)\} & \text{if }d>q.
    \end{cases}\]
    WLP holds for \eqref{eq:KT-mod-Td} if and only if for every $q=p^k\geq p$ we have
    \begin{equation}\label{eq:conds-xiq}
    (1)\ \xi_q \leq\frac{2q-2}{n},\quad\text{ or }\quad(2)\
        \xi_q=\theta_q(d-1-q),\text{ $n$ is odd, and }\xi_q \leq\frac{q}{n-2}.
    \end{equation}
\end{theorem}

\begin{proof} We fix $q=p^k$, let $\ul{a}$ as in \eqref{eq:ai=d-1}, and use notation \eqref{eq:d-1=2mq+dbar}. By Theorem~\ref{thm:new-WLP}, it suffices to prove that condition \eqref{eq:thetaq-leq-gap} holds for every $q$-reduction $\ul{r}$ of $\ul{a}$ if and only if \eqref{eq:conds-xiq} holds.

Suppose first that \eqref{eq:thetaq-leq-gap} holds. If $\xi_q=\theta_q(d-1)$ then we consider the $q$-reduction $r_i=\ol{d}$ which is even and balanced, so $g_+(\ul{a};\ul{r})=0$. The periodicity of the function $\theta_q$ implies $\xi_q=\theta_q(d-1)=\theta_q(\ol{d})$, so \eqref{eq:thetaq-leq-gap} yields
\[n\cdot\xi_q = \sum_{i=1}^n\theta_q(r_i) \leq \left(g_+(\ul{a};\ul{r})+2\right)q-2 = 2q-2,\]
proving condition (1) in \eqref{eq:conds-xiq}. We may therefore assume that $\xi_q=\theta_q(d-1-q)$ and $d>q$. If $n$ is even then taking $r_i = \ol{d}\pm q$ as in \eqref{eq:possible-aibar} we obtain again an even balanced reduction, hence $n\cdot\xi_q\leq 2q-2$ follows as before from \eqref{eq:thetaq-leq-gap}. If instead $n$ is odd then $\xi_q \leq\frac{q}{n-2}$ is trivial when $n=3$, so we may assume $n\geq 5$. It follows that $\ol{d}\not\in\{q-1,2q-1\}$, since otherwise the $q$-reduction \eqref{eq:qred-allq-1} would not satisfy \eqref{eq:thetaq-leq-gap}. We can then define an even balanced reduction by $r_1 = \ol{d}$ and $r_i = \ol{d}\pm q$ for $i>1$, and using \eqref{eq:theta-dbar-pm-q} we obtain
\[q-2+(n-2)\xi_q = (\theta_q(r_1)+\theta_q(r_2)) + \sum_{i=3}^n\theta_q(r_i) =  \sum_{i=1}^n\theta_q(r_i) \leq 2q-2,\]
which implies $\xi_q\leq q/(n-2)$, as desired.

We assume now that \eqref{eq:conds-xiq} holds, consider a $q$-reduction $\ul{r}$ of $\ul{a}$, and proceed to verifying \eqref{eq:thetaq-leq-gap}. Using \eqref{eq:possible-aibar} and the periodicity of the function $\theta_q$, it follows that $\theta_q(r_i)\leq \xi_q$. If $\xi_q\leq(2q-2)/n$ then
\[\sum_{i=1}^n\theta_q(r_i) \leq n\cdot\xi_q \leq 2q-2 \leq \left(g_+(\ul{a};\ul{r})+2\right)q-2,\]
as desired. Suppose instead that condition (2) in \eqref{eq:conds-xiq} holds. If $\ul{r}$ is an even reduction then using \eqref{eq:possible-aibar} and the fact that $n$ is odd, there exist an odd number of indices for which $r_i = \ol{d}$. Since $\theta_q(\ol{d})+\xi_q = \theta_q(\ol{d})+\theta_q(\ol{d}-q)$ is equal to either $-\infty$ or to $(q-2)$ by \eqref{eq:theta-dbar-pm-q}, we obtain
\[\sum_{i=1}^n\theta_q(r_i) \leq \theta_q(\ol{d}) + (n-1)\xi_q \leq (q-2)+(n-2)\xi_q \leq 2q-2 \leq \left(g_+(\ul{a};\ul{r})+2\right)q-2. \]
If $\ul{r}$ is an odd reduction then $g_+(\ul{a};\ul{r})\geq 1$ so \eqref{eq:thetaq-leq-gap} holds when $n=3$. We may therefore assume that $n\geq 4$, in which case we have
\[\sum_{i=1}^n\theta_q(r_i) \leq n\cdot\xi_q \leq \frac{n\cdot q}{n-2}\leq 3q-2,\]
where the last inequality is equivalent to $2(n-2)\leq q(2n-6)$, which holds because $q\geq 2$ and $2n-6\geq n-2$. Since $3\leq g_+(\ul{a};\ul{r})+2$, it follows that \eqref{eq:thetaq-leq-gap} holds, concluding our proof.
\end{proof}

\subsection{Classification of WLP in characteristic $p=2$.}
\label{subsec:WLP-char2-constantai}

If $\op{char}(\kk)=2$ and $n\geq 4$, it follows from \cite{kus-vra}*{Remark~5.2 and Theorem~6.3} that \eqref{eq:KT-mod-Td} fails WLP for all $d\geq 2$. To see this, consider $q=2^k\geq 2$ with
\[ \frac{q}{2}\leq d-1<q.\]
It follows that
\[\frac{2q+n-1}{n}< \frac{q}{2}+1\leq d\leq q,\]
hence WLP fails by \eqref{eq:Md-easyfail-WLP}. The only interesting case occurs then when $n=3$, where the classification of WLP was obtained by Brenner and Kaid.

\begin{proposition}[\cite{bre-kai}*{Corollary~2.7}]\label{prop:WLPddd}
    If $p=2$ and $n=3$, then WLP holds for \eqref{eq:KT-mod-Td} if and only if 
    \[d=\left\lfloor \frac{2^t+1}{3}\right\rfloor\quad\text{ for some }t.\]
\end{proposition}

\begin{proof}
    We consider $k$ such that $2^k\leq d-1<2^{k+1}$. It follows that \[2^{k-1}\leq\frac{a_1\oplus a_2\oplus a_3}{2}=\frac{d-1}{2}\quad\text{and}\quad 2\cdot\max(a_1,a_2,a_3)=2(d-1)<2^{k+2},\]
    so it suffices to consider $q=2^k$ and $q=2^{k+1}$ in Theorem~\ref{thm:WLP-char-2}.

    If $q=2^k$ then $\theta_q(a_i) = 2q-1-d$ for $d\neq 2q$ and $\theta_q(a_i)=-\infty$ for $d=2q$. Since $2q\geq(4q-1)/3$, we get
    \[\sum_{i=1}^n \theta_q(a_i) \leq 2q-2 \quad\Longleftrightarrow\quad d\geq\frac{4q-1}{3}=\frac{2^{k+2}-1}{3}.\]
    If $q=2^{k+1}$ then $\theta_q(a_i)=d-1$, hence
    \[\sum_{i=1}^n \theta_q(a_i) \leq 2q-2 \quad\Longleftrightarrow\quad d\leq\frac{2q+1}{3}=\frac{2^{k+2}+1}{3}.\]
    Combining the above inequalities with Theorem~\ref{thm:WLP-char-2} we obtain the desired conclusion with $t=k+2$.
\end{proof}

\subsection{WLP in $n=3$ variables.} We prove the following simple characterization for the failure of WLP, and explain how it is a reformulation of \cite{bre-kai}*{Theorem~2.6}.

\begin{theorem}\label{thm:WLP-n=3}
    If $n=3$ then WLP fails for \eqref{eq:KT-mod-Td} if and only if there exists $q=p^k\geq p$ and $m\geq 0$ such that
    \begin{equation}\label{eq:ineq-WLP-n=3}
    |d-(2m+1)q| \leq \frac{q-2}{3}.
    \end{equation}
\end{theorem}

When $d$ is even, $(q-2)$ has the same parity as $3\cdot|d-(2m+1)q|$, so \eqref{eq:ineq-WLP-n=3} is equivalent to
\[ |d-(2m+1)q|<\frac{q}{3},\text{ or }-\frac{q}{3}<d-(2m+1)q<\frac{q}{3}\]
which can be rewritten as
\[ \frac{3d}{6m+2}>q>\frac{3d}{6m+4}.\]
When $d$ is odd, we can rewrite \eqref{eq:ineq-WLP-n=3} as
\[ |d-(2m+1)q|<\frac{q-1}{3},\text{ or equivalently }\frac{3d-1}{6m+2}>q>\frac{3d+1}{6m+4}.\]
It follows that Theorem~\ref{thm:WLP-n=3} is equivalent to \cite{bre-kai}*{Theorem~2.6}.

\begin{proof}[Proof of Theorem~\ref{thm:WLP-n=3}]
    Applying Theorem~\ref{thm:WLP-constant-ai}, the failure of WLP is equivalent to the existence of $q=p^k\geq p$ such that \eqref{eq:conds-xiq} fails. Since $n=3$ is odd and $q/(n-2)=q>\xi_q$, condition (2) in \eqref{eq:conds-xiq} is equivalent to the equality $\xi_q=\theta_q(d-1-q)$. It follows that \eqref{eq:conds-xiq} fails if and only if $\xi_q=\theta_q(d-1)\geq(2q-1)/3$. 

    Note that if $\theta_q(d-1)\geq(2q-1)/3$ then either $\theta_q(d-1-q)=-\infty$ or 
    \[\theta_q(d-1-q)=q-2-\theta_q(d-1)\leq(q-5)/3\leq(2q-1)/3\leq\theta_q(d-1),\] 
    hence $\xi_q=\theta_q(d-1)$. Using notation \eqref{eq:d-1=2mq+dbar}, we have $\theta_q(d-1) = \theta_q(\ol{d})$, and it follows that WLP fails if and only if $\theta_q(\ol{d})\geq(2q-1)/3$. This is further equivalent to
    \[ \frac{2q-1}{3}\leq \ol{d}=d-1-2mq \leq 2q-2-\frac{2q-1}{3}.\]
    Adding $1-q$ to the above chain of inequalities we obtain \eqref{eq:ineq-WLP-n=3}, which concludes our proof.
\end{proof}

\subsection{WLP in $n=4$ variables.} The characteristic $2$ case was addressed in Section~\ref{subsec:WLP-char2-constantai}, so in order to complete the classification of WLP in $4$ variables, we verify the following result due to Kustin and Vraciu.

\begin{theorem}[\cite{kus-vra}*{Theorem 5.1}]
 If $d\geq 2$, $p\geq 3$ and $n=4$, then \eqref{eq:KT-mod-Td} satisfies WLP if and only if there exist integers $a,k,r$ such that
 \[d = a\cdot p^k+r,\text{ with }1\leq a\leq\frac{p-1}{2},\ r=\frac{p^k\pm 1}{2}.\]
\end{theorem}

\begin{proof} If $d\leq p$ then $\xi_p=d-1$ and since $n=4$ is even, Theorem~\ref{thm:WLP-constant-ai} implies that WLP holds if and only if $d-1\leq(p-1)/2$. If WLP holds then we can take $a=d-1$, $k=0$ and $r=1$ to get the desired expression for $d$. Conversely, the expression above for $d$ can only hold if $k=0$ and $d=a$ or $d=a+1$, and the condition $a\leq (p-1)/2$ implies $d-1\leq(p-1)/2$, which yields WLP.

We may therefore assume $d>p$, and find $k\geq 1$ such that $p^k\leq d<p^{k+1}$. We write $d=a\cdot p^k+r$, with $1\leq a\leq p-1$ and $0\leq r\leq p^k-1$. Applying Theorem~\ref{thm:WLP-constant-ai}, we have to check that condition \eqref{eq:conds-xiq} holds for all $q$ if and only if $a\leq(p-1)/2$ and $r=(p^k\pm 1)/2$. Since $n$ is even, condition \eqref{eq:conds-xiq} is equivalent to $\xi_q\leq(2q-2)/n=(q-1)/2$. 

If $q\geq d$ then $\xi_q=d-1$, so condition \eqref{eq:conds-xiq} for all $q\geq p^{k+1}$ is equivalent to
\[ d \leq \frac{p^{k+1}+1}{2} = \frac{p-1}{2}\cdot p^k + \frac{p^k+1}{2},\]
which implies $a\leq(p-1)/2$. We may therefore assume from now on that $1\leq a\leq(p-1)/2$.

If $p\leq q\leq p^k$ then $\xi_q=q-1$ if $q|d$ and otherwise $\xi_q\geq(q-2)/2$ by \eqref{eq:theta-dbar-pm-q}. It follows that the condition $\xi_q\leq(q-1)/2$ is equivalent to $\xi_q=(q-1)/2$ since $q$ is odd. Using notation \eqref{eq:d-1=2mq+dbar}, $\xi_q=(q-1)/2$ if and only if $\ol{d}\in\left\{\frac{q-3}{2},\frac{q-1}{2},q+\frac{q-3}{2},q+\frac{q-1}{2}\right\}$, which is then equivalent to
\begin{equation}\label{eq:d-equiv-modq-forWLP} 
d=2mq+1+\ol{d} \equiv \frac{q\pm 1}{2} \ (\op{mod}\ q).
\end{equation}
Applying this with $q=p^k$ we obtain $r=(q\pm 1)/2$, which in turn implies \eqref{eq:d-equiv-modq-forWLP} for all $p\leq q\leq p^k$.
\end{proof}

\subsection{WLP in $n\geq 5$ variables.} 

As discussed in Section~\ref{subsec:WLP-char2-constantai}, if $p=2$ and $n\geq 5$ then WLP fails for all $d\geq 2$. We next show that if $p\geq 3$ then WLP holds only if we are in the situation of Theorem~\ref{thm:WLP-from-socle}(1).

\begin{theorem}[\cite{kus-vra}*{Theorem 6.4}]
    If $n\geq 5$, $p \geq 3$, then \eqref{eq:KT-mod-Td} satisfies WLP if and only if
    \[ d \leq \frac{2p+n-2}{n}.\]
\end{theorem}

\begin{proof}
    Since $M_{\ul{a}}$ has socle degree $s=n\cdot(d-1)$, it follows from Theorem~\ref{thm:WLP-from-socle}(1) that WLP holds when $d\leq(2p+n-2)/n$. We then have to verify that WLP fails if $d\geq(2p+n-1)/n$. If $d\leq p$, this is a consequence of \eqref{eq:Md-easyfail-WLP}, while for $p|d$ it follows from \eqref{eq:WLP-p-divides-d}. We may then further assume that $d\geq p+1$, $p\nmid d$.

    Using the notation in Theorem~\ref{thm:WLP-constant-ai} and \eqref{eq:theta-dbar-pm-q}, it follows that $\xi_p\geq(p-2)/2$, and since $p$ is odd, we have $\xi_p\geq(p-1)/2$. Since $n\geq 5$ we have
    \[\xi_p\geq\frac{p-1}{2}>\frac{2p-2}{n},\]
    so condition (1) in \eqref{eq:conds-xiq} cannot hold for $q=p$. If $(n,p)\neq(5,3)$ we have moreover that
    \[\xi_p\geq\frac{p-1}{2} > \frac{p}{n-2},\]
    so condition (2) in \eqref{eq:conds-xiq} also fails for $q=p$ and $(n,p)\neq (5,3)$. 
    
    If $(n,p)=(5,3)$ then in order for condition (2) in \eqref{eq:conds-xiq} to hold for $q=p$, we must have $\xi_p = \theta_p(d-1-p)=1$, which forces $d\geq 5$. Taking $q=p^2=9$ we either have $d\leq q$ and $\xi_q = d-1\geq 4$, or $\xi_q\geq(q-2)/2 = 7/2$. In either case we get $\xi_q > (2q-2)/n=16/5$ and $\xi_q>q/(n-2)=3$, so \eqref{eq:conds-xiq} does not hold and WLP fails, concluding our proof.
\end{proof}

\section{The strong Lefschetz property}
\label{sec:SLP}

In this section, we focus on the strong Lefschetz property (SLP) and give an alternative proof of the characterization in \cites{nick,lund-nick}. Our approach leverages the graded Green--Han--Monsky representation ring discussed in \cite{KMRR}, and considers SLP in the context of modules \cite{harima-watanabe}. Following \cite{green}, \cite{han-monsky}, we consider the category $\mc{G}$ of finite length graded $\kk[T]$-modules $M$, and define the tensor product $M\oo_{\kk} N$ by letting $T$ act~via
\[ T \cdot (m\oo n) = Tm\oo n + m\oo Tn.\]
We write $\Delta$ for the split Grothendieck ring of $\mc{G}$, and refer to $\Delta$ as the \defi{graded Green--Han--Monsky (GHM) representation ring} (see \cite{KMRR}*{Section~3} for some basic properties). A $\bb{Z}$-basis for $\Delta$ consists of the indecomposable modules $\delta_c(-j)$ in $\mc{G}$, where $\delta_c=\kk[T]/(T^c)$ for $c\geq 1$, and $j\in\bb{Z}$ indicates the degree shift for the cyclic generator  (we will abuse notation and use the same symbol for an object in $\mc{G}$ and its class in $\Delta$).

If $1\leq a\leq b$ then there exist positive integers $c_j=c_j(a,b)$, $0\leq j<a$, which depend on $\op{char}(\kk)$, such that
\begin{equation}\label{eq:dela-delb-general} \delta_a\oo_{\kk}\delta_b=\bigoplus_{j=0}^{a-1}\delta_{c_j}(-j).
\end{equation}
In the ring $\Delta$ we will simply write $\delta_a\delta_b$ for the class of the tensor product. When convenient, we work in the ungraded category $\mc{G}^u$ and the corresponding ring $\Delta^u$, where degree shifts are ignored--in \eqref{eq:dela-delb-general} they can be recovered from the ungraded multiplication using $c_0 \geq c_1 \geq \cdots \geq c_{a-1}$ (see \cite{KMRR}*{Proposition~3.1}). 

We say that the decomposition \eqref{eq:dela-delb-general} is \defi{standard} (see also \cite{StandardJordan}) if
\begin{equation}\label{eq:standard-cj}
    c_j=a+b-1-2j\text{ for }j=0,\cdots,a-1,
\end{equation}
which is the decomposition when $\op{char}(\kk)=0$. Our goal is to explain how SLP for the monomial complete intersections \eqref{eq:def-Ma} translates in terms of standard decompositions, and show how Renaud's work \cite{renaud} leads to a classification of when SLP holds.

\subsection{SLP and standard products in the GHM ring}
\label{subsec:SLP=standardGHM}

For a non-zero module $M$ in $\mc{G}$ we write 
\[M=\bigoplus_{i\in\bb{Z}} M_i\] 
and define the \defi{average degree of $M$} as (note that $M_i\neq 0$ for finitely many values of $i$)
\begin{equation}\label{eq:def-avg-deg}
    \avg(M) = \frac{\sum_{i\in\bb{Z}} i\cdot \dim_{\kk}(M_i)}{\dim_{\kk}M}.
\end{equation}
It is a weighted average of the dimensions of the graded components of $M$, and whenever we have a direct sum decomposition $M=M^1\oplus\cdots\oplus M^r$, with $M,M^j\in\mc{G}$, we obtain
\begin{equation}\label{eq:av=weightedsum}
    \avg(M) =\frac{\sum_{j=1}^r \avg(M^j)\cdot\dim_{\kk}(M^j)}{\dim_{\kk}M}.
\end{equation}
Moreover, for $M',M''\in\mc{G}$, their tensor product satisfies
\[\avg(M'\oo_\kk M'') = \avg(M')+\avg(M'').\]
For the indecomposable modules $\delta_c(-j)$ we have
\[\avg(\delta_c(-j)) = j+\frac{c-1}{2}.\]
Following the terminology in \cite{harima-watanabe}*{Section~2}, we say that $M$ is \defi{symmetric} if
\[\dim_{\kk}M_i = \dim_{\kk}M_{s-i}\quad\text{ for all }i\in\bb{Z},\text{ where }s=2\avg(M),\]
and note that for such modules $s/2=\avg(M)$ is called the 
\defi{reflecting degree} of $M$ in \cite{harima-watanabe}. A symmetric module $M$ has the \defi{strong Lefschetz property (SLP)} if for each $i\leq\avg(M)$ we have ($s=2\avg(M)$ as above)
\[\times T^{s-2i} : M_i \lra M_{s-i} \text{ is an isomorphism}.\]
The modules $\delta_c(-j)$ are symmetric, and they satisfy SLP with $s=c+2j-1$. 

Expressing a module $M\in\mc{G}$ as a sum of indecomposables $\delta_c(-j)$ (possibly with multiplicities), it follows from the definition (see also \cite{harima-watanabe}*{Lemma~3.3}) that 
\begin{equation}\label{eq:SLP=samecenter}
M\text{ has SLP }\Longleftrightarrow\avg(M) = \avg(\delta_c(-j))\text{ for each summand }\delta_c(-j)\text{ of }M.
\end{equation}
It follows from \eqref{eq:av=weightedsum} and \eqref{eq:SLP=samecenter} that for symmetric modules $M$ without SLP there must exist summands $\delta_{c_1}(-j_1)$ and $\delta_{c_2}(-j_2)$ of $M$ satisfying
\[\avg(\delta_{c_1}(-j_1)) < \avg(M) < \avg(\delta_{c_2}(-j_2)).\]

If $M',M''$ are symmetric then so is $M'\oo_{\kk} M''$, but SLP is usually not preserved under tensor products. Our interest is in the SLP property for the module $M_{\ul{a}} = \delta_{1+a_1}\oo \cdots \oo \delta_{1+a_n}$, which can be identified with the algebra \eqref{eq:def-Ma} by writing $\delta_{1+a_i}=\kk[T_i]/(T_i^{1+a_i})$ and letting $T=T_1+\cdots+T_n$. It follows from the discussion in Section~\ref{subsec:conventions} that SLP for the algebra \eqref{eq:def-Ma} can be tested on the linear form $\ell=T$, hence it is the same as SLP for the $\kk[T]$-module $M_{\ul{a}}$. Note that
\[2\avg(M_{\ul{a}}) = a_1+\cdots+a_n\]
is the socle degree of $M_{\ul{a}}$. Applying \eqref{eq:SLP=samecenter} to $M_{\ul{a}}$ we obtain
\begin{equation}\label{eq:SLP-forMa}
M_{\ul{a}}\text{ has SLP }\Longleftrightarrow c+2j-1=a_1+\cdots+a_n\text{ for each summand }\delta_c(-j)\text{ of }M_{\ul{a}}.
\end{equation}

\begin{theorem}\label{thm:tensor-SLP}
    Suppose that $M=M'\oo_{\kk} M''$, with $M',M''\in\mc{G}$. The following are equivalent:
    \begin{enumerate}
        \item $M$ has SLP.
        \item $M'$ and $M''$ have SLP, and for all summands $\delta_{c'}(-j')$ of $M'$, and $\delta_{c''}(-j'')$ of $M''$, $\delta_{c'}\oo\delta_{c''}$ has SLP.
    \end{enumerate}
\end{theorem}

\begin{proof} Write $r=\avg(M)$ and $r'=\avg(M')$, $r''=\avg(M'')$, so $r=r'+r''$.  

\noindent $(1)\Rightarrow (2).$ If $M'$ failed SLP, it would have an indecomposable summand $U'$ with $\avg(U')<r'$.  Choose a summand $U''$ of $M''$ with $\avg(U'')\leq r''$, and then a summand $U$ of $U'\otimes_{\kk} U''$ with
\[
 \avg(U)\leq \avg(U')+\avg(U'')<r'+r''=r.
\]
This contradicts \eqref{eq:SLP=samecenter} for $M$.  Hence $M'$ has SLP, and similarly $M''$ has SLP.  It follows that every indecomposable summand of $M'$ and $M''$ has average $r'$ and $r''$, respectively.  If some $U'\otimes_{\kk} U''$ failed SLP, it would contain a summand of average strictly less than $r'+r''=r$, again a contradiction. Since degree shifts don't change SLP, we get that $\delta_{c'}\oo\delta_{c''}$ has SLP, as desired.

\noindent $(2)\Rightarrow (1).$  Every indecomposable summand $U$ of $M$ occurs in $U'\otimes_{\kk} U''$ for some indecomposable summands $U'\subset M'$ and $U''\subset M''$. By hypothesis $U'\otimes_{\kk} U''$ has SLP, so the average of $U$ is
\[
\avg(U)= \avg(U'\otimes_{\kk} U'')=\avg(U')+\avg(U'')=r'+r''=r,
\]
and (6.5) implies that SLP holds for $M$.
\end{proof}

\begin{corollary}\label{cor:inductive-SLP}
    Fix $\ul{a}\in\bb{Z}^n_{\geq 0}$, with $a_i\geq 0$ for $1\leq i\leq n$. The following are equivalent:
    \begin{enumerate}
        \item $M_{\ul{a}}$ has SLP.
        \item $M_{(a_1,\cdots,a_{n-1})}$ has SLP, and for each summand $\delta_{c'}(-j')$ of $M_{(a_1,\cdots,a_{n-1})}$, $M_{(c'-1,a_n)}$ has SLP.
    \end{enumerate}
\end{corollary}

\begin{proof}
    The conclusion follows by applying Theorem~\ref{thm:tensor-SLP} with $M=M_{\ul{a}}$, $M'=M_{(a_1,\cdots,a_{n-1})}$, and $M''=\delta_{1+a_n}$, and using the fact that $M_{(c'-1,a_n)}=\delta_{c'}\oo \delta_{1+a_n}$.
\end{proof}

The equality on the right-hand side of \eqref{eq:SLP-forMa} always holds in characteristic zero, and more generally for summands $\delta_c(-j)$ where $c$ is coprime to the characteristic $p$ \cite{KMRR}*{Proposition~3.3}. It is the summands $\delta_c$ where $p|c$ that will detect the failure of SLP, and we illustrate this with a simple scenario.

\begin{corollary}\label{cor:failSLP-pdivis}
    If $n\geq 2$, $\op{char}(\kk)=p\geq 3$, $a_1,\cdots,a_n\geq 1$, and $p|a_i+1$ for some $i$, then SLP fails for $M_{\ul{a}}$.
\end{corollary}

\begin{proof} We may assume $p|a_1+1$, hence every summand $\delta_{c'}(-j)$ of $M_{(a_1,\cdots,a_{n-1})}$ satisfies $p|c'$ \cite{han-monsky}*{Theorem~3.6}. If $M_{\ul{a}}$ has SLP, then Corollary~\ref{cor:inductive-SLP} implies that so does $M_{(c'-1,a_n)}$. By \eqref{eq:dela-delb-general}, \eqref{eq:SLP-forMa}, we have 
\[M_{(c'-1,a_n)} = \delta_{c'}\oo \delta_{1+a_n} = \delta_{c'+a_n}+\delta_{c'+a_n-2}(-1)+\cdots\]
Since $p\mid c'$, every length of a summand in $\delta_{c'}\oo\delta_{1+a_n}$ is divisible by $p$. We get $p|c'+a_n$ and $p|c'+a_n-2$, so $p$ divides their difference which equals $2$, contradicting the fact that $p\geq 3$.    
\end{proof}

\subsection{Renaud's algorithm}
\label{subsec:renaud-algo}

In this section we recall Renaud's algorithm \cite{renaud}*{Section~3} for computing the (ungraded) product $\delta_a\delta_b$ in characteristic $p>0$. If $1\leq a\leq b$ then the characteristic zero expression is
\begin{equation}\label{eq:dela-delb-char0}
\delta_a\delta_b = \delta_{b-a+1}+\delta_{b-a+3}+\cdots+\delta_{b+a-1}=\sum_{i=0}^{a-1} \delta_{b-a+1+2i}.
\end{equation}

\begin{definition}\label{def:adjustments}
    Fix $1\leq a\leq b$, let $q=p^k$, $k\geq 1$, and let $m\geq 1$ with $\gcd(m,p)=1$. We write
    \[ a = a_0q+a_1,\quad b=b_0q+b_1,\quad\text{with } 0\leq a_1,b_1<q,\]
    and define 
    \[\eta_q(m) = \begin{cases}
        |a_1-b_1| & \text{if }m\equiv a_0+b_0\ (\op{mod }2),\\
        |q-(a_1+b_1)| & \text{if }m\not\equiv a_0+b_0\ (\op{mod }2).\\
    \end{cases}\]
    Let $\bb{E}$ be a linear combination of indecomposable summands $\delta_c$, such as the
    right-hand side of \eqref{eq:dela-delb-char0}. We say that a
    \defi{$(q,m)$-adjustment occurs in $\bb{E}$} if $\eta=\eta_q(m)\geq 2$ and one of the following holds:
    \begin{enumerate}
        \item $\eta$ is even and both $\delta_{mq+1},\delta_{mq-1}$ appear in $\bb{E}$.
        \item $\eta$ is odd and $\delta_{mq}$ appears in $\bb{E}$.
    \end{enumerate}
    In this case the \defi{$(q,m)$-adjustment} of $\bb{E}$ replaces the collection of summands $\delta_c$ with $mq-\eta<c<mq+\eta$, by $\eta\cdot\delta_{mq}$. We call $q$ the \defi{adjustment level} and $mq$ the \defi{adjustment center}.
\end{definition}

Renaud's algorithm starts with the characteristic-zero decomposition
\eqref{eq:dela-delb-char0} and performs a sequence of adjustments, ultimately producing the decomposition in characteristic $p$. More precisely, let $\gamma$ be minimal such that $p^\gamma>a,b$. The algorithm first performs all possible adjustments at level $p^\gamma$
(in any order), then all possible adjustments at level $p^{\gamma-1}$, and continues successively through the levels $p^{\gamma-2},\ldots,p$. We note that the original formulation of Renaud's algorithm allows for $\eta\in\{0,1\}$, but the resulting adjustment does not change the expression $\bb{E}$ so we omit this possibility in our setup.  

\begin{remark}\label{rem:Renaud}
(1) Any adjustment produces a summand with multiplicity $\geq 2$, and any further adjustments will continue to have such summands. It follows that \eqref{eq:dela-delb-char0} is the characteristic $p$ expression for the product $\delta_a\delta_b$ if and only if no adjustments occur when running Renaud's algorithm.

\noindent (2) The center $mq$ for any adjustment in Renaud's algorithm must lie strictly between $b-a+1$ and $b+a-1$. Indeed, it follows from Definition~\ref{def:adjustments} that $b-a+1\leq mq\leq b+a-1$, so we must verify that the inequalities are strict. If $mq=b-a+1$ then either $m=b_0-a_0$ and $b_1=a_1-1$, or $m=b_0-a_0+1$ and $(a_1,b_1)=(0,q-1)$, and in both cases we have $\eta=1$. Similarly, if $mq=b+a-1$ then either $m=b_0+a_0$ and $a_1+b_1=1$, or $m=b_0+a_0+1$ and $b_1+a_1-1=q$, forcing again $\eta=1$.

\noindent (3) In Definition~\ref{def:adjustments} we have $a+b-mq = (a_0+b_0-m)\cdot q+(a_1+b_1)$, which has the same parity as $\eta$. It follows that for the expression $\bb{E}$ in \eqref{eq:dela-delb-char0}, if the adjustment center $mq$ lies strictly between $b-a+1$ and $b+a-1$ then one of the conditions (1), (2) in Definition~\ref{def:adjustments} must be satisfied. If this is the case then a $(q,m)$-adjustment occurs if and only if $\eta\geq 2$.
\end{remark}

\begin{example}\label{ex:smalla-p=2}
    Suppose that $p=2$. For $a=2\leq b$ we have
    \[ \delta_2\delta_b = \begin{cases}
        2\cdot\delta_{b} & \text{if }b=2r,\\
        \delta_{b-1} + \delta_{b+1} & \text{if }b=2r+1.\\
    \end{cases}\]
    In the case $b=2r$, we write $r=2^{k-1}\cdot m$ with $k\geq 1$, $m$ odd, and the desired formula follows by applying a $(2^k,m)$-adjustment to the characteristic zero expression $\delta_{b-1}+\delta_{b+1}$. In the case $b=2r+1$ is odd, no adjustments can occur. A similar analysis shows that for $a=3\leq b$: if we write $r=2^{k-2}\cdot m$ with $m$ odd, $k\geq 2$, then we have
    \[
    \delta_3\delta_b = 
    \begin{cases}
        3\cdot \delta_{b} & \text{if }b=4r,\text{ after a $(2^k,m)$-adjustment},\\
        2\cdot \delta_{b-1}+\delta_{b+2} & \text{if }b=4r+1,\text{ after a $(2^k,m)$-adjustment},\\
        \delta_{b-2}+\delta_b+\delta_{b+2} & \text{if }b=4r+2,\text{ no adjustment},\\
        \delta_{b-2}+2\cdot \delta_{b+1} & \text{if }b=4r-1,\text{ after a $(2^k,m)$-adjustment}.\\
    \end{cases}
    \]
\end{example}

For $q\geq 1$ and $r\in\bb{Z}$ we consider the minimum distance from $r$ to a multiple of $q$:
\begin{equation}\label{eq:def-qdist}
    d_q(r) = \min_{t\in\bb{Z}}|r-tq|,
\end{equation}
noting that for all $r$ we have $0\leq d_q(r)\leq\lfloor q/2\rfloor$.

\begin{lemma}\label{lem:large-q-noadjustment}
    Consider the characteristic zero expression $\bb{E}$ in \eqref{eq:dela-delb-char0}, $1\leq a\leq b$. The following are equivalent:
    \begin{enumerate}
        \item $d_q(b) \geq a-1$ for all prime powers $q=p^k>a$.
        \item no adjustments occur in $\bb{E}$ for prime powers $q=p^k>a$.
    \end{enumerate}
\end{lemma}

\begin{proof}
    $(1) \Rightarrow (2).$ Suppose that a $(q,m)$-adjustment occurs for $q=p^k>a$. With the notation in Definition~\ref{def:adjustments} we have $a_0=0$, $a_1=a$, and $d_q(b)=\min(b_1,q-b_1)$. The condition $d_q(b)\geq a-1$ implies $b-a+1\geq b_0\cdot q$ and $b+a-1\leq(b_0+1)\cdot q$. This contradicts the inequality $b-a+1<mq<b+a-1$ from Remark~\ref{rem:Renaud}(2).

\noindent $(2) \Rightarrow (1).$ Suppose that $d_q(b)\leq a-2$ (in particular $a\geq 2$) and use the notation in Definition~\ref{def:adjustments}. If $d_q(b)=b_1$ then $b-a+1<b_0\cdot q<b+a-1$. We write $b_0 = m\cdot p^r$ with $p\nmid m$, and set $Q = p^r\cdot q$. We then have $b = m\cdot Q + b_1$, $a=0\cdot Q + a$, and $\eta_Q(m)=|a-b_1|\geq 2$, so a $(Q,m)$-adjustment occurs. If instead $d_q(b)=q-b_1$ then $b-a+1<(b_0+1)\cdot q<b+a-1$. We write $b_0+1 = m\cdot p^r$ with $p\nmid m$, and set $Q = p^r\cdot q$. We obtain $b = (m-1)\cdot Q + (Q-q+b_1)$, $a=0\cdot Q+a$, hence $\eta_Q(m)=|Q-(Q-q+b_1+a)|=|(q-b_1)-a|\geq 2$ and a $(Q,m)$-adjustment occurs again, concluding our proof.
\end{proof}

\subsection{SLP in two variables}
\label{subsec:SLP-2vars}
We can now give a quick proof of the characterization of SLP in $n=2$ variables. In characteristic $p=2$, we have the following (see \cite{Cook12}*{Corollary~4.8} and \cite{StandardJordan}*{Theorem~1}). 

\begin{theorem}\label{thm:SLP-n=p=2}
    Suppose that $\op{char}(\kk)=2$ and that $2\leq a\leq b$. The algebra $\kk[T_1,T_2]/\langle T_1^a,T_2^b\rangle$ has SLP if and only if one of the following holds:
    \begin{enumerate}
        \item $a=2$ and $b$ is odd.
        \item $a=3$ and $b\equiv 2\text{ (mod 4)}$.
    \end{enumerate}
\end{theorem}

\begin{proof}
    If $a=2,3$ then the desired conclusion follows from Example~\ref{ex:smalla-p=2}, so we may assume that $a\geq 4$. We let $q=2^k\geq 4$ such that $q\leq a<2q$, and note that by Lemma~\ref{lem:large-q-noadjustment} we have $d_{2q}(b)\geq a-1\geq q-1$. Since $d_{2q}(b)\leq q$, it follows that $a\in\{q,q+1\}$. If $a=q$ then every summand of $\delta_a\delta_b$ has the form $\delta_{cq}$ (see \cite{han-monsky}*{Theorem~3.6}) so \eqref{eq:dela-delb-char0} cannot hold (see also the proof of Corollary~\ref{cor:failSLP-pdivis}). It follows that $a=q+1$ and $d_{2q}(b)=q$, so $q|b$ and we obtain again that every summand of $\delta_a\delta_b$ has the form $\delta_{cq}$, a contradiction.
\end{proof}

For odd primes we obtain the following characterization, which is equivalent to \cite{nick}*{Theorem~3.2} and \cite{StandardJordan}*{Theorems~2 and~3}.

\begin{theorem}\label{thm:SLP-n=2-p>2}
    Suppose that $\op{char}(\kk)=p>2$ and that $2\leq a\leq b$. Let $q=p^k$, $k\geq 0$, such that $q\leq a<pq$. The algebra $\kk[T_1,T_2]/\langle T_1^a,T_2^b\rangle$ has SLP if and only if $d_{pq}(b)\geq a-1$ and one of the following holds:
    \begin{enumerate}
        \item $a<p$.
        \item $a\geq p$, and $d_q(a)=d_q(b)=(q-1)/2$ (or equivalently, $a,b\equiv (q\pm 1)/2\text{ (mod $q$)}$).
    \end{enumerate}
\end{theorem}

\begin{proof}
    If $a<p$ then $q=1$ and the desired conclusion follows from Lemma~\ref{lem:large-q-noadjustment}. We may therefore assume that $q\geq p$, and that $d_{pq}(b)\geq a-1$ (which is a requirement for SLP by Lemma~\ref{lem:large-q-noadjustment}). We may further assume that $p\nmid a,b$ since otherwise SLP fails by Corollary~\ref{cor:failSLP-pdivis}, and the condition $a,b\equiv (q\pm 1)/2\text{ (mod $q$)}$ fails as well.
    
    If we write $a=a_0q+a_1$, $b=b_0q+b_1$ as in Definition~\ref{def:adjustments}, then $1\leq a_0\leq b_0$ and $1\leq a_1,b_1<q$, hence 
    \[ b-a+1=(b_0-a_0)q+b_1-a_1+1 < (b_0-a_0+1)q\quad\text{ and }\quad (b_0+a_0)q < a+b-1.\]
    Since $b_0-a_0+1<b_0+a_0$, we can find two consecutive multiples of $q$ satisfying $b-a+1<mq<(m+1)q<b+a-1$. The condition $d_{pq}(b)\geq a-1$ implies that no multiple of $pq$ lies strictly between $b-a+1$ and $b+a-1$, so $m$ and $m+1$ are coprime to $p$. 

    If SLP holds, then no $(q,m)$- or $(q,m+1)$-adjustment can occur, so by Remark~\ref{rem:Renaud} we get $\eta_q(m),\eta_q(m+1)\leq 1$. This implies $|a_1-b_1|,|q-(a_1+b_1)|\leq 1$, which is only possible when $a_1,b_1\in\{(q\pm 1)/2\}$.

    Conversely, if $a_1,b_1\in\{(q\pm 1)/2\}$ then for every prime power $p\leq q'\leq q$, if we write $a=a'_0 q'+a_1'$, $b=b_0'q'+b_1'$, with $0\leq a_1',b_1'<q'$, then $a_1',b_1'\in\{(q'\pm 1)/2\}$. This means that $\eta_{q'}(m)\leq 1$ for all $q'$ and all $m$, hence no adjustment can occur at level $q'\leq q$. For prime powers $q'\geq pq$ we have $d_{q'}(b)\geq d_{pq}(b)\geq a-1$, so Lemma~\ref{lem:large-q-noadjustment} shows no adjustments can occur at levels $q'>q$ either, and therefore SLP holds.
\end{proof}

One quick consequence of Theorems~\ref{thm:SLP-n=p=2},~\ref{thm:SLP-n=2-p>2} is that if $1\leq a_1, a_2<p$ then
\begin{equation}\label{eq:Ma12-small-SLP}
    M_{(a_1,a_2)}\text{ has SLP }\Longleftrightarrow a_1+a_2<p.
\end{equation}
Indeed, if $p=2$ then we must have $a_1=a_2=1$ and both sides of the equivalence are false. If $p>2$ then we may assume that $a_1\leq a_2$ and let $a=1+a_1$, $b=1+a_2$ in Theorem~\ref{thm:SLP-n=2-p>2}. If $a=p$ then SLP fails, and $a_1+a_2\geq a\geq p$, so both sides of \eqref{eq:Ma12-small-SLP} are false. If $a<p$ then SLP is equivalent to $d_p(b)\geq a-1$, which in turn is equivalent to $p-b\geq a-1$ (since $b\geq a-1$), or $a_1+a_2=a+b-2\leq p-1$, proving \eqref{eq:Ma12-small-SLP}.

\subsection{SLP in $n\geq 3$ variables}
\label{subsec:SLP-manyvars}
The following reproves \cite{Cook12}*{Corollary~6.3}, \cite{lund-nick}*{Theorem~3.8}.

\begin{theorem}\label{thm:SLP-manyvars}
    If $n\geq 3$ and $1\leq a_1\leq\cdots\leq a_n$ then $M_{\ul{a}}$ has SLP if and only if one of the following holds:
    \begin{enumerate}
        \item $a_1+\cdots+a_n<p$.
        \item $a_n\geq p$ and $a_1+\cdots+a_{n-1}\leq d_p(a_n+1)$.
    \end{enumerate}
    In particular, SLP fails for $p=2$.
\end{theorem}

\begin{proof}
    Once we prove the desired characterization of SLP, the last statement is explained by the fact that conditions (1) and (2) fail in characteristic $2$: we have $a_1+\cdots+a_n\geq a_1+\cdots+a_{n-1}\geq n-1\geq 2>d_2(a_n+1)$.

    Suppose that (1) or (2) holds, and in particular $p\geq 3$. Since $d_p(a_{n}+1)\leq \lfloor p/2\rfloor\leq p-2$, it follows that $a_1+\cdots+a_{n-1}<p$. By induction on $n$ (and using Theorem~\ref{thm:SLP-n=2-p>2} when $n=3$) it follows that SLP holds for $(a_1,\cdots,a_{n-1})$. Using \eqref{eq:SLP-forMa}, every summand $\delta_{c'}(-j')$ of $M_{(a_1,\cdots,a_{n-1})}$ satisfies
    \[c'-1\leq c'+2j'-1=a_1+\cdots+a_{n-1}\leq p-2,\]
    where the last inequality follows from $a_n\geq 1$ in case (1), and from $d_p(a_n+1)\leq p-2$ in case (2). It follows that $c'<p$, and therefore $M_{(c'-1,a_n)}$ satisfies SLP by Theorem~\ref{thm:SLP-n=2-p>2}. Corollary~\ref{cor:inductive-SLP} implies SLP for~$M_{\ul{a}}$.

    Assume now that SLP holds for $M_{\ul{a}}$, and apply Theorem~\ref{thm:tensor-SLP} to conclude that SLP holds for $M_{(a_{i_1},\cdots,a_{i_k})}$ for any  $1\leq i_1<\cdots<i_k\leq n$. If $p=2$, Theorem~\ref{thm:SLP-n=p=2} applied to $M_{(a_1,a_3)}$ and $M_{(a_2,a_3)}$ yields $a_1,a_2\in\{1,2\}$. Since $M_{(1,1)}$ and $M_{(2,2)}$ fail SLP, we only have to consider the case $(a_1,a_2)=(1,2)$. SLP for $M_{(a_1,a_3)}$ implies then that $a_3$ is even, and SLP for $M_{(a_2,a_3)}$ implies that $a_3$ is odd, a contradiction. 
    
    It remains to analyze the case when $p\geq 3$. If $a_{n-1}<p$ then it follows by induction when $n>3$, and by \eqref{eq:Ma12-small-SLP} when $n=3$, that $a_1+\cdots+a_{n-1}<p$. This implies that $\delta_c$ is a summand of $M_{(a_1,\cdots,a_{n-1})}$, where $c=a_1+\cdots+a_{n-1}+1$, while Corollary~\ref{cor:inductive-SLP} implies that $M_{(c-1,a_n)}$ has SLP. If $a_n<p$ then \eqref{eq:Ma12-small-SLP} implies $c-1+a_n<p$ and condition (1) holds. If $a_n\geq p$ then Theorem~\ref{thm:SLP-n=2-p>2} implies $d_p(a_n+1)\geq c-1$ and condition (2) holds. We are then left with the case $a_{n-1}\geq p$, and since $M_{(a_{n-2},a_{n-1},a_n)}$ has SLP, we may assume further that $n=3$. Since $M_{(a_2,a_3)}$ has SLP and $p\leq a_2\leq a_3$, it follows from \eqref{eq:dela-delb-char0} with $a=1+a_2$, $b=1+a_3$ that one of the summands of $M_{(a_2,a_3)}$ is of the form $\delta_c(-j)$ with $p|c$. Therefore $M_{(a_1,c-1)}$ fails SLP by Corollary~\ref{cor:failSLP-pdivis}, and so does $M_{\ul{a}}$ by Corollary~\ref{cor:inductive-SLP}.
\end{proof}

\section*{Acknowledgements}
Experiments with Macaulay2 \cite{GS} and discussions with ChatGPT \cite{ChatGPT} have provided many valuable insights. Marangone gratefully acknowledges that this research was supported in part by the Pacific Institute for the Mathematical Sciences. Marangone acknowledges the European Union’s Horizon Europe research and innovation program under the Marie Skłodowska-Curie Funding Program (Project PrIMes, Grant Agreement No. 101277460). Raicu and Reed acknowledge the support of the National Science Foundation Grant DMS-2302341. Reed also acknowledges support by by the National Natural Science Foundation of China (No. 12288201). Part of the material in this paper is based upon work supported by the National Science Foundation under Grant No. DMS-1928930 and by the Alfred P. Sloan Foundation under grant G-2021-16778, while Raicu and Reed were in residence at the Simons Laufer Mathematical Sciences Institute (formerly MSRI) in Berkeley, California, during the Spring 2024 semester.

\end{document}